\documentclass[11pt]{amsart}

\usepackage{amsmath,amssymb,amsthm, mathrsfs}
\usepackage{mathtools}
\usepackage[colorinlistoftodos]{todonotes}
\usepackage{tikz}
\usepackage{tikz-cd}
\usetikzlibrary{positioning}
\usetikzlibrary{matrix,arrows,decorations.pathmorphing}

\usepackage{tikzit}

\tikzstyle{red dot}=[fill={rgb,255: red,255; green,37; blue,102}, draw=black, shape=circle]
\tikzstyle{green}=[fill={rgb,255: red,81; green,255; blue,7}, draw=black, shape=circle, tikzit category=green dot]
\tikzstyle{orange}=[fill={rgb,255: red,255; green,128; blue,0}, draw={rgb,255: red,255; green,128; blue,0}, shape=circle]
\tikzstyle{For labels}=[fill=white, draw=red, shape=circle]
\tikzstyle{FOR LABEL-GREEN}=[fill=white, draw={rgb,255: red,28; green,179; blue,11}, shape=circle]
\tikzstyle{blue}=[fill=blue, draw=blue, shape=circle]
\tikzstyle{Rectangle}=[fill=white, draw=red, shape=rectangle]

\tikzstyle{new edge style 0}=[->]

\DeclareMathOperator*{\bigast}{\text{\LARGE{\textasteriskcentered}}}
\newcommand{\nclose}[1]{\ensuremath{\langle\!\langle#1\rangle\!\rangle}}

\newtheorem{thm}{Theorem}[section]

\newtheorem{exmp}{Example}[section]
\newtheorem{theorem}[thm]{Theorem} \newtheorem{proposition}[thm]{Proposition}  
\newtheorem{lemma}[thm]{Lemma}

\newtheorem{corollary}[thm]{Corollary}

\theoremstyle{definition}

\newtheorem{definition}[thm]{Definition}
\newtheorem{remark}[thm]{Remark}

\newcommand{\msf}{\mathsf}

\title[]{Relative Dehn fuctions, hyperbolically embedded subgroups and combination theorems}
\author{Hadi Bigdely}
\author{Eduardo Mart\'inez-Pedroza }

\date{\today}

\begin{document}

\maketitle

\begin{abstract}
Consider the following classes of pairs consisting of a group and a finite collection of subgroups:
\begin{itemize}
\item \( \mathcal{C}= \left\{ (G,\mathcal H) \mid \text{$\mathcal{H}$ is hyperbolically embedded in $G$} \right\} \) 

\item \( \mathcal{D}= \left\{ (G,\mathcal H) \mid \text{the relative Dehn function of $(G,\mathcal H)$ is well-defined} \right\} .\) 
\end{itemize}
Let $G$ be a group that splits as a finite graph of groups such that each vertex group $G_v$ is assigned a finite collection of subgroups $\mathcal{H}_v$, and each edge group $G_e$ is conjugate to a subgroup of some $H\in \mathcal{H}_v$ if $e$ is adjacent to $v$.
 Then there is a finite collection of subgroups $\mathcal{H}$ of $G$ such that 
\begin{enumerate}
    \item If each $(G_v, \mathcal{H}_v)$ is in $\mathcal C$, then $(G,\mathcal{H})$ is in $\mathcal C$.
 
    \item If each $(G_v, \mathcal{H}_v)$ is in $\mathcal D$, then $(G,\mathcal{H})$ is in $\mathcal D$.
    
    \item For any vertex $v$ and for any $g\in G_v$, the element $g$ is conjugate to an element in some $Q\in\mathcal{H}_v$ if and only if $g$ is conjugate to an element in some $H\in\mathcal{H}$.
\end{enumerate}
That edge groups are not assumed to be finitely generated and that they do not necessarily belong to a peripheral collection of subgroups of an adjacent vertex are the main differences between this work and previous results in the literature.  The method of proof provides lower and upper bounds of the relative Dehn functions in terms of the relative Dehn functions of the vertex groups. These bounds  generalize and improve analogous results in the literature.
\end{abstract}


\section{Introduction}

Consider the following classes of pairs consisting of a group and a finite collection of subgroups:
\begin{itemize}
\item \( \mathcal{C}= \left\{ (G,\mathcal H) \mid \text{$\mathcal{H}$ is hyperbolically embedded in $G$} \right\} \) 

\item \( \mathcal{D}= \left\{ (G,\mathcal H) \mid \text{the relative Dehn function of $(G,\mathcal H)$ is well-defined} \right\} .\) \end{itemize}

\begin{theorem}\label{thmx:general}
Let $G$ be a group that splits as a finite graph of groups such that each vertex group $G_v$ is assigned a finite collection of subgroups $\mathcal{H}_v$, and each edge group $G_e$ is conjugate to a subgroup of some $H\in \mathcal{H}_v$ if $e$ is adjacent to $v$.
 Then there is a finite collection of subgroups $\mathcal{H}$ of $G$ such that 
\begin{enumerate}
    \item If each $(G_v, \mathcal{H}_v)$ is in $\mathcal C$, then $(G,\mathcal{H})$ is in $\mathcal C$.
 
    \item If each $(G_v, \mathcal{H}_v)$ is in $\mathcal D$, then $(G,\mathcal{H})$ is in $\mathcal D$.
    
    \item For any vertex $v$ and for any $g\in G_v$, the element $g$ is conjugate in $G_v$ to an element of some $Q\in\mathcal{H}_v$ if and only if $g$ is conjugate in $G$ to an element of some $H\in\mathcal{H}$.
\end{enumerate}
\end{theorem}

The theorem is trivial without the third item in the conclusion; indeed, the pair $(G, \{G\})$ belongs to both $\mathcal{C}$ and $\mathcal{D}$.  
In comparison with previous results in the literature, our main contribution is that our combination results do not assume that edge groups are finitely generated or contained in $\mathcal H_v$.

The notion of a hyperbolically embedded collection of subgroups was introduced by 
Dahmani, Guirardel and Osin~\cite{DGO17}. 
A pair $(G,\mathcal H)$ in $\mathcal{C}$ is called a \emph{hyperbolically embedded pair} and we write $\mathcal H\hookrightarrow_h G$. Our combination results for  hyperbolically embedded pairs $(G,\mathcal H)$ generalize  analogous results for relatively hyperbolic pairs in~\cite{Dah03,Ali05,Os06b, MjLa08, BW13} and  for hyperbolically embedded pairs~\cite{DGO17,MiOs15}.

The notions of finite relative presentation and  relative Dehn function $\Delta_{G, \mathcal H}$ of a group $G$ with respect to a collection of subgroups $\mathcal H$ were introduced by Osin~\cite{Osin06} generalizing the notions of finite presentation and Dehn function of a group.
A pair $(G,\mathcal H)$ is called \emph{finitely presented} if $G$ is finitely presented relative to $\mathcal H$, and $\Delta_{G,\mathcal H}$ is called the \emph{Dehn function of the pair} $(G,\mathcal H)$.
While a finitely presented group has a well-defined Dehn function, in contrast, the  Dehn function of a finitely presented pair $(G,\mathcal H)$ is not always well-defined, for a characterization see~\cite[Thm.E(2)]{HuMaSa21}. 
Our result generalizes combination results  for pairs $(G,\mathcal H)$ with well-defined  Dehn function by Osin~\cite[Thms. 1.2 and 1.3]{Os06b}. 

We prove Theorem~\ref{thmx:general} for the case of graphs of groups with a single edge, since then the general case follows directly by induction on the number of edges of the graph. This particular case  splits into three subcases corresponding to the three results stated below. The proofs of these subcases use characterizations of pairs $(G,\mathcal H)$ being hyperbolically embedded~\cite[Thm. 5.9]{MR21}, and having a well-defined Dehn function~\cite[Thm. 4.7]{HuMaSa21} in terms of existence of $G$-graphs with certain properties that relate to Bowditch's fineness~\cite{Bo12}. These characterizations  are discussed in Section~\ref{sec:02}. The proof of  Theorem~\ref{thmx:general} for the case of a graph of groups with a single edge entails the construction of graphs satisfying the conditions of those characterizations  for the fundamental group of the graph of groups. We use the existing graphs for the vertex groups as building blocks. 

Our method of proof provides lower and upper bounds for the relative Dehn function of the fundamental group of the graph of groups in the terms of the relative Dehn functions of the vertex groups, see Section~\ref{sec:addendum}. Specifically,
Theorem~\ref{cor:DehnFunctions} below  generalizes results of Brick~\cite{Br93} on bounds for the Dehn functions of free products (see the improvement by Guba and Sapir~\cite{GuSa99}) and improve the bounds found by Osin for relative Dehn functions in~\cite[Thms 1.2 and 1.3]{Osin06}.

Our main result reduces to the following statements. 

\begin{theorem}[Amalgamated Product]\label{thmx:Io}
 For $i\in\{
 1,2\}$, let $(G_i, \mathcal{H}_i\cup\{K_i\})$ be a pair and   $\partial_i\colon C\to K_i$  a group monomorphism. Let $G_1\ast_C G_2$  denote the amalgamated product determined by
$G_1\xleftarrow{\partial_1}C\xrightarrow{\partial_2}G_2$, and let  $\mathcal{H}=\mathcal{H}_1\cup\mathcal{H}_2$. Then:
\begin{enumerate}
    \item   \label{item:Ia} If $\mathcal{H}_i\cup\{K_i\} \hookrightarrow_h G_i$ for each $i$, then $\mathcal H\cup \{\langle K_1,K_2\rangle\}  \hookrightarrow_h G_1\ast_C G_2$.  
  
    \item \label{item:Ib}  If    $( G_i,\mathcal{H}_i\cup\{K_i\} ) \in \mathcal D $   for each $i$, then $(G_1\ast_C G_2, \mathcal{H}  \cup\{\langle K_1,K_2 \rangle\}) \in \mathcal D$.  
    
\item For any $g\in G_i$, the element $g$ is conjugate in $G_i$ to an element of some $Q\in\mathcal{H}_i\cup \{K_i\}$ if and only if $g$ is conjugate in $G$ to an element of some $H\in\mathcal{H}\cup \{ \langle K_1, K_2 \rangle\}$.    
\end{enumerate}
\end{theorem}

In the following statements, for a subgroup $K$ of a group $G$ and an element $g\in G$, the conjugate subgroup $gKg^{-1}$ is denoted by $K^g$.

\begin{thm}[HNN-extension I] \label{thmx:IIIo}
Let $(G, \mathcal{H} \cup\{K,L\})$ be a pair with $K\neq L$,  $C$  a  subgroup of $K$, and   $\varphi\colon C\to L$ a group monomorphism. Let $G\ast_{\varphi}$ denote the HNN-extension
$\langle G, t\mid  t c t^{-1} =\varphi(c)~\text{for all $c\in C$} \rangle$. Then:
 \begin{enumerate}
    \item \label{item:IIIa}  If $\mathcal{H} \cup\{K,L\}\hookrightarrow_h G$ then $\mathcal H \cup\{\langle K^t, L\rangle\} \hookrightarrow_h G\ast_{\varphi}$.
    \item \label{item:IIIb}  If $(G, \mathcal{H}\cup\{K,L\}) \in \mathcal{D}$, then $(G\ast_{\varphi},\mathcal H \cup\{\langle K^t, L\rangle\}) \in \mathcal D$.
    \item For any $g\in G$, the element $g$ is conjugate in $G$ to an element of some $Q\in\mathcal{H} \cup \{K, L\}$ if and only if $g$ is conjugate in $G\ast_\varphi$ to an element of some $H\in\mathcal{H}\cup \{ \langle K^t, L \rangle\}$. 
\end{enumerate}
\end{thm}

Note that the third items of Theorems~\ref{thmx:Io}
and~\ref{thmx:IIIo}  follow directly from standard arguments in combinatorial group theory. This article focuses on proving the other statements. 

\begin{corollary}[HNN-extension II]\label{thmx:II}
Let $(G, \mathcal{H} \cup\{K\})$ be a pair,  $C$ a  subgroup of $K$, $s\in G$, and $\varphi\colon C\to K^s$ a group monomorphism. Let $G\ast_{\varphi}$ denote the HNN-extension
$\langle G, t\mid  t c t^{-1} =\varphi(c)~\text{for all $c\in C$} \rangle$. Then:
\begin{enumerate}
    \item   \label{item:IIab} 
     If $\mathcal{H} \cup\{K\}\hookrightarrow_h G$ then $\mathcal H \cup\{\langle K, s^{-1}t\rangle\} \hookrightarrow_h G\ast_{\varphi}$.
    \item \label{item:IIbb} 
     If $(G, \mathcal{H}\cup\{K\}) \in \mathcal{D}$, then $(G\ast_{\varphi},\mathcal H \cup\{\langle K, s^{-1}t\rangle\}) \in \mathcal D$.
     \item For any $g\in G$, the element $g$ is conjugate in $G$ to an element of some $Q\in\mathcal{H} \cup \{K\}$ if and only if $g$ is conjugate in $G\ast_\varphi$ to an element of some $H\in\mathcal{H}\cup \{ \langle K, s^{-1}t \rangle\}$.
\end{enumerate}
\end{corollary}
\begin{proof}
First we prove the statement in the case that $s$ is the identity element of $G$.  Let $L$ be the HNN-extension 
$L=K\ast_\varphi$. Observe that there is a natural isomorphism between  $G\ast_{\varphi}
$ and the amalgamated product $G\ast_{K} L$. In this case, the conclusion of the corollary is obtained directly by invoking Theorem~\ref{thmx:Io}, since the pair $(L, \{L\} )$ is in both classes $\mathcal C$ and $\mathcal D$. 

Now we argue in the case that $s\in G$ is arbitrary. Let $\psi\colon C \to K$ the composition $I_s\circ \varphi$ where $I_s$ is the inner automorphism $I_s(x)=s^{-1}xs$.
Since
\[G\ast_\varphi=\langle G, t \mid c^{s^{-1}t}=\varphi(c)^{s^{-1}} \text{ for all  $c\in C$} \rangle,
\]
there is a natural isomorphism $G\ast_\varphi \to G\ast_\psi$ which restricts to the identity on the base group $G$, and the stable letter of $G\ast_\psi$ corresponds to $s^{-1}t$ in $G\ast_\varphi$. Since $\psi$ maps $C\leq K$ into $K$, we have reduced the case of arbitrary $s\in G$ to the case that $s$ is the identity in $G$ and the statement of the corollary follows.
\end{proof}

Let us describe the argument proving our main result using the three previous statements. The argument  relies on the following observation.

\begin{remark}\label{rem:ConjugatingParabolics}
If a pair $(G,\mathcal{H}\cup \{L\})$ belongs to $\mathcal C$ (respectively $\mathcal D$) and $g\in G$  then $(G,\mathcal{H}\cup\{L^g\})$ belongs to $\mathcal C$ (respectively $\mathcal D$).   This statement can be seen directly from the original definitions of hyperbolically embedded collection of subgroups~\cite{DGO17}, and relative Dehn function~\cite{Osin06}. It can be also deduced directly from Theorems~\ref{thm:Farhan} and~\ref{thm:DehnFunctionFine} respectively in the main body of the article.     
\end{remark}

\begin{proof}[Proof of Theorem~\ref{thmx:general}]
The case of a tree of groups satisfying the hypothesis of the theorem follows from Theorem~\ref{thmx:Io} and Remark~\ref{rem:ConjugatingParabolics}. Then the general case reduces to the case of a graph of groups with a single vertex, where the vertex group corresponds to the fundamental group of a maximal tree of groups.  In the case of a graph of groups with a single vertex, each edge corresponds to applying either  Theorem~\ref{thmx:IIIo} or  Corollary~\ref{thmx:II} together with  Remark~\ref{rem:ConjugatingParabolics}.
\end{proof}

The following theorem generalizes results of Brick~\cite[Proposition 3.2]{Br93} on bounds on Dehn functions of free products, and improve bounds for relative Dehn functions found by Osin~\cite[Theorems 1.2 and 1.3]{Osin06}.

\begin{theorem}\label{cor:DehnFunctions}
\begin{enumerate}
\item Under the assumptions of Theorem~\ref{thmx:Io}\eqref{item:Ib}, if $\Delta$ is a relative Dehn function of  $(G_1\ast_C G_2, \mathcal{H}  \cup\{\langle K_1,K_2 \rangle\})$ and   $\Delta_i$ is a relative Dehn function of  $(G_i, \mathcal{H}_i  \cup\{K_i\})$ then 
\[ \max\{\Delta_1,\Delta_2\} \preceq \Delta  \preceq  \max\left\{\overline{\Delta_1} , \overline{\Delta_2} \right\},\] 
 where $\overline{\Delta_i}$ denotes the super-additive closure of $\Delta_i$.
 \item Under the assumptions of Theorem~\ref{thmx:IIIo}\eqref{item:IIIb}, if $\Delta$ is a relative Dehn function of  $(G\ast_\varphi, \mathcal{H}  \cup\{\langle K^t,L \rangle\})$ and   $\Delta_0$ is a relative Dehn function of  $(G , \mathcal{H}   \cup\{K, L\})$ then \[ \Delta_0 \preceq \Delta  \preceq  \overline{\Delta_0 } ,\] 
 where $\overline{\Delta_0}$ is the super-additive closure of $\Delta_0$
 \item Under the assumptions of Corollary~\ref{thmx:II}\eqref{item:IIbb}, if $\Delta$ is a relative Dehn function of  $(G\ast_{\varphi},\mathcal H \cup\{\langle K, s^{-1}t\rangle\})$ and   $\Delta_0$ is a relative Dehn function of  $(G, \mathcal{H}\cup\{K\})$ then 
\[ \Delta_0 \preceq \Delta  \preceq  \overline{\Delta_0} ,\] 
 where $\overline{\Delta_0}$ is the super-additive closure of $\Delta_0$
 \end{enumerate}
 \end{theorem}

We conclude the introduction with a more detailed comparison of our results with  previous results in the literature.
\begin{enumerate} 
\item Dahmani, Guirardel and Osin proved  Theorem~\ref{thmx:Io}\eqref{item:Ia} in the case that $\partial_1\colon C\to K_1$ is an isomorphism and $K_1$ is finitely generated~\cite[Thm 6.20]{DGO17}; and   Theorem~\ref{thmx:IIIo}\eqref{item:IIIa} in the case that $C=K$ and  $K$ is finitely generated~\cite[Thm 6.19]{DGO17}.  
\item Osin proved    Theorem~\ref{thmx:Io}\eqref{item:Ib} in the case that $\partial_1\colon C\to K_1$ is an isomorphism and $K_1$ is finitely generated, see~\cite[Thm 1.3]{Os06b}; and   Theorem~\ref{thmx:IIIo}\eqref{item:IIIb} in the case that $C=K$ and $K$ is finitely generated, see~\cite[Thm 1.2]{Os06b}.

\item Under the assumptions of Theorem~\ref{thmx:general}, if  each $(G_v, \mathcal{H}_v)\in \mathcal C$ for every vertex $v$, and there is at least one $v$ such that $\mathcal{H}_v$ is nontrivial in $G_v$, the existence of a nontrivial collection $\mathcal{H}$ such that $(G, \mathcal{H})\in \mathcal{C}$
follows from 
results of Minasyan and Osin~\cite[Cor. 2.2 and 2.3]{MiOs15} and the characterization of acylindrical hyperbolicity in terms of existence of proper infinite hyperbolically embedded subgroups by  Osin~\cite{OsinAcylindrical}; by a nontrivial collection we mean that it contains a proper infinite subgroup. This alternative approach does not guarantee that the collection $\mathcal H$   satisfies the third condition of Theorem~\ref{thmx:general}. 

\item Theorems~\ref{thmx:Io}\eqref{item:Ia} and ~\ref{thmx:IIIo}\eqref{item:IIIa}, in the case that $G_i$ is hyperbolic relative to $\mathcal{H}_i$ for $i=1,2$, follow from results of Wise and the first author~\cite[Thm. A]{BW13}. 
\end{enumerate}

\subsection*{Organization.} The rest of the article consists of five sections. In Section~\ref{sec:02} we review characterizations of pairs $(G, \mathcal H)$ being hyperbolically embedded and having well-defined Dehn functions in terms of actions on graphs.  In Section~\ref{sec:03}, we reduce the proof of Theorems~\ref{thmx:Io} and ~\ref{thmx:IIIo} to prove two technical results, Theorems~\ref{thm:CombinationFine} and~\ref{thm:HNNFine}. Their proofs are the content of Sections~\ref{sec:04} and~\ref{sec:05} respectively. The last section contains the proof of Theorem~\ref{cor:DehnFunctions}.

\subsection*{Acknowledgments.} We thank the referee for feedback and corrections. 
The authors also thank Sam Hughes for comments in a preliminary version of the article. The first  author acknowledges funding by the Fonds de Recherche du Québec–Nature et Technologies  FRQNT. 
 The second author acknowledges funding by the Natural Sciences and Engineering Research Council of Canada NSERC. 

\section{Characterizations using Fineness}\label{sec:02}

In this section,  we describe a  characterization  of pairs $(G,\mathcal H)$ being hyperbolically embedded, Theorem~\ref{thm:Farhan}; and a characterization of the pairs having a well-defined Dehn function, Theorem~\ref{thm:DehnFunctionFine}. These characterizations are  in terms of existence of $G$-graphs with certain properties that relate to Bowditch's fineness~\cite{Bo12}, a notion that is defined below. The characterizations are re-statements of  previous results in the literature~\cite[Thm. 5.9]{MR21} and \cite[Thm. 4.7]{HuMaSa21}. This section also includes a couple of lemmas that will be of use  in later sections. 

All graphs $\Gamma=(V,E)$ considered in this section are simplicial, so we consider the set of edges $E$ to be a  collection of subsets of cardinality two of the vertex set $V$.

Let $\Gamma$ be a simplicial graph, let $v$ be a vertex of $\Gamma$, and let $T_v\Gamma$ denote the set of the vertices adjacent to $v$. For $x, y \in T_v\Gamma$, the angle metric $\angle_v(x, y)$ is the combinatorial length of the shortest path in the graph $\Gamma-\{v\}$ between $x$ and $y$, with $\angle_v(x, y)=\infty$ if there is no such path. The graph $\Gamma$ is fine at $v$ if $(T_v\Gamma, \angle_v)$ is a locally finite metric space.   A graph is \emph{fine} if it is fine at every vertex. 
 
It is an observation that a graph $\Gamma$ is fine if and only if for every pair of vertices $x,y$ and every positive integer $n$, there are finitely many embedded paths between $x$ and $y$ of length at most $n$; for a proof see~\cite{Bo12}.

\subsection{Hyperbolically embedded pairs} 
In~\cite[Definition  2.9]{OsinAcylindrical}, Osin defines the notion of  a collection of subgroups $\mathcal{H}$ being hyperbolically embedded into a group $G$. This relation is denoted as $\mathcal{H} \hookrightarrow_h G$ and, in this case, we say that the pair $(G,\mathcal H)$ is a hyperbolically embedded pair. In this article we use the following characterization of hyperbolically embedded collection proved in~\cite{MR21} as our working definition.

\begin{definition}[Proper pair]
 A  pair $(G,\mathcal{H})$ is \emph{proper} if $\mathcal H$ is a finite collection of subgroups such that no two distinct infinite  subgroups are conjugate in $G$.
\end{definition}

\begin{theorem}[Criterion for hyperbolically embedded pairs]\label{thm:Farhan}\emph{\cite[Theorem 5.9]{MR21}} 
A proper pair $(G,\mathcal{H})$ is a hyperbolically embedded pair  if and only if 
there is a connected $G$-graph   $\Gamma$ such that
\begin{enumerate}
\item  There are finitely many $G$-orbits of vertices.
  \item Edge $G$-stabilizers are finite.
    
    \item Vertex $G$-stabilizers are either finite or conjugates of subgroups in $\mathcal{H}$.
    
    \item Every $H\in \mathcal{H}$ is the $G$-stabilizer of a vertex of $\Gamma$.
    
\item $\Gamma$ is hyperbolic.
\item $\Gamma$ is fine at $V_{\infty}(\Gamma)=\{v\in V(\Gamma)~|~v~\text{has infinite stabilizer}\}$.
\end{enumerate}
\end{theorem}

\begin{definition}
We refer to a graph $\Gamma$ satisfying the conditions of Theorem~\ref{thm:Farhan}  as a  \emph{$(G, \mathcal{H})$-graph}
\end{definition}

Let us observe that in~\cite{MR21}, Theorem~\ref{thm:Farhan} is proved for the case that $\mathcal H$ consists of a single infinite subgroup, and the authors observe that the argument in the case that $\mathcal H$ is a finite collection of infinite subgroups (such that no pair of distinct infinite subgroups in $\mathcal H$ are conjugate in $G$) follows by the same argument. Then the general case in which $\mathcal{H}$ is a finite collection of subgroups follows from the following statement: If $\mathcal H$ is a collection of subgroups and $K$ a finite subgroup of a group $G$ then:
\begin{enumerate}
    \item  
     $\mathcal{H}\hookrightarrow_h G$ if and only if  $\mathcal{H}\cup\{K\} \hookrightarrow_h G$.
     \item There is $(G,\mathcal H)$-graph if and only if  there is a $(G,\mathcal{H}\cup\{K\})$-graph.
\end{enumerate}
The first statement is a direct consequence of the definition of hyperbolically embedded collection by Osin~\cite{OsinAcylindrical}. The if part of the second statement is trivial, and the only if part follows directly from~\cite[Thm. 3.4]{ArMP22}. 

\subsection{Relative presentations  }

In~\cite[Chapter 2]{Osin06}, Osin introduces the notions of relative presentation of a group with respect to a collection of subgroups, and relative Dehn functions. We briefly recall these notions below.

Let $G$ be a group and let $\mathcal{H}$ be a collection of subgroups.
A subset $S$  of $G$ is a \emph{relative generating set} of $G$ with respect to $\mathcal H$ if the natural homomorphism
\begin{equation}\label{eq:rel:presentation}
    F(S,\mathcal{H})=F(S)\ast \bigast_{H\in \mathcal{H}}H\longrightarrow G
\end{equation}
is surjective, where $F(S)$ denotes the free group with free generating set $S$.
A relative generating set of $G$ with respect to $\mathcal H$ is called a \emph{generating set of the pair $(G, \mathcal H)$}. A pair that admits a finite generating set is called a \emph{finitely generated pair}.  Let $R\subseteq F(S,\mathcal{H})$ be a subset that normally generates the kernel of the above homomorphism. In this case, we have a short exact sequence of groups
\[ 1\to \nclose{R}\to F(S,\mathcal{H}) \to G \to 1,\]
and the triple
\begin{equation}\label{G:relative:presentation}
    \langle S,\mathcal{H}\ |\ R \rangle
\end{equation}
is called a \emph{relative presentation of $G$ with respect to $\mathcal H$}, or  just a \emph{presentation of the pair $(G,\mathcal H)$}. Abusing notation, we write $G=\langle S, \mathcal{H}\mid R \rangle$. If both $S$ and $R$ are finite we say that the pair $(G,\mathcal{H})$ is \emph{finitely presented}. 

\begin{lemma}\label{lem:Observation}
Let $G$ be a group and let    $\mathcal{H}_0\sqcup \mathcal{H}$ be a collection of subgroups. Let $P$ denote the subgroup of $G$ generated by $S_0$ and the subgroups in $\mathcal{H}_0$.
If 
\[
G= \langle S_0\sqcup S, \mathcal{H}_0\cup\mathcal{H} \mid R_0\sqcup R \rangle \quad\text{and}\quad P = \langle S_0,  \mathcal{H}_0 \mid R_0 \rangle
\]
then 
\[
G= \left\langle  S, \mathcal{H}\cup\left\{ P \right\} \mid R' \right\rangle,
\]
where $R'$ is the image of $R$ under the natural epimorphism $\varphi\colon F(S_0\cup S, \mathcal{H}_0\cup\mathcal{H}) \to F(S,\mathcal{H}\cup \{ P\})$
\end{lemma}

\begin{proof}
Let  $A=F(S, \mathcal{H})$, $B=F(S_0, \mathcal{H}_0)$, $K$ the normal subgroup of $B$ generated by $R_0$ and $N$ the normal subgroup of $A\ast B=F(S_0\cup S , \mathcal{H}_0\cup \mathcal{H})$ generated by $R$. Our hypotheses imply that the natural epimorphisms $A\ast B \to G$ and $B\to P$ induce short exact sequences \[ 1 \to \nclose{N,K}\to   A\ast B \to G \to 1,\quad\text{and}\quad 1 \to K \to B \to P \to 1.\]
Let us identify $P=B/K$.
The natural epimorphism of the statement of the lemma \[\varphi\colon A\ast B \to A \ast (B/K)\] induces an isomorphism 
\[\hat\varphi\colon \frac{A\ast B}{\nclose{N,K}} \to \frac{A\ast (B/K)}{  \varphi(N)} = \frac{A\ast P}{\varphi(N)}.\]
By the definition of $N$, we have that  $\varphi(N)$ is the normal subgroup of $A\ast P$ generated by $R'=\varphi(R)$. Therefore the natural epimorphism $A\ast P \to G$ induces a short exact sequence
\[ 1 \to \nclose{R'}\to A\ast P \to G \to 1 \]
which concludes the proof. 
\end{proof}

The following pair of lemmas allow us to conclude that certain amalgamated products and HNN-extensions preserve  relative finite presentability.

\begin{lemma}[Amalgamated Products]\label{prop:AmalgamatedPresentation}
For $i\in\{
 1,2\}$, let $(G_i, \mathcal{H}_i\cup\{K_i\})$ be a pair,   $\partial_i\colon C\to K_i$  a group monomorphism. Let $G_1\ast_C G_2$  denote the amalgamated product determined by
$G_1\xleftarrow{\partial_1}C\xrightarrow{\partial_2}G_2$, and  $\mathcal{H}=\mathcal{H}_1\cup\mathcal{H}_2$. 
If \[G_i = \left\langle S_i, \mathcal{H}_i\cup\{K_i\}\mid R_i \right\rangle\] then \[G_1\ast_C G_2 = \left\langle S_1\cup S_2, \mathcal{H}\cup\{ \langle K_1,K_2 \rangle\}  \mid R_1\cup R_2 \right\rangle.\] 
\end{lemma}
\begin{proof}
Observe that $\langle S_1\cup S_2, \mathcal{H}\cup\{K_1, K_2\} \mid R_1\cup R_2, \partial_1(c)=\partial_2(c) \text{ for all $c\in C$} \rangle$ is a relative presentation of $G_1\ast_C G_2$. Since the subgroup $\langle K_1, K_2 \rangle \leq G_1\ast_C G_2$ is isomorphic to the amalgamated product   $K_1\ast_C K_2$, we have that  $\langle  K_1,K_2  \mid \partial_1(c)=\partial_2(c) \text{ for all $c\in C$} \rangle$ is a relative presentation of $\langle K_1,K_2 \rangle$. The proof concludes by invoking Lemma~\ref{lem:Observation}.
\end{proof}

\begin{lemma}[HNN-extension]\label{prop:RelPreHNN}
Let $(G, \mathcal{H} \cup\{K,L\})$ be a pair with $K\neq L$,  $C$  a  subgroup of $K$,    $\varphi\colon C\to L$ a group monomorphism, and let  $G\ast_{\varphi}$ denote the HNN-extension
$\langle G, t\mid  t c t^{-1} =\varphi(c)~\text{for all $c\in C$} \rangle$.
If \[G=\langle S, \mathcal{H} \cup\{K,L\} \mid R \rangle\]
then 
\[G\ast_{\varphi}=\left\langle S,t, \mathcal{H}\cup\{ \langle K^t, L \rangle\} \mid R' \right \rangle, \]
where $R'$ is the set of relations obtained by taking each element of  $R$ and replacing all occurrences of elements $k\in K$ by  words $t^{-1} k^t t$. In particular, $R$ and $R'$ have the same cardinality.
\end{lemma}
\begin{proof} 
Let $J$ denote the subgroup $K^t$, and let $\psi\colon K \to J$ be the isomorphism $\psi(k)=tkt^{-1}$. Observe that $\langle S, t, \mathcal{H}\cup\{K,L\} \mid R,\   tct^{-1}=\varphi(c) \text{ for all $c\in C$} \rangle$ is a  presentation for the pair $(G\ast_\varphi, \mathcal{H}\cup\{K,L\})$.  Therefore   
\[ G\ast_\varphi = \langle S, t, \mathcal{H}\cup\{J,L\} \mid R',\   \psi(c)=\varphi(c) \text{ for all $c\in C$} \rangle.\]  A consequence of Britton's lemma is that the subgroup $\langle J, L \rangle \leq G\ast_\varphi$ is isomorphic to the amalgamated product $J\ast_{\varphi(C)}L$. Hence,  
\[  \langle J, L\rangle  =  \langle \{J,L\} \mid   \psi(c)=\varphi(c) \text{ for all $c\in C$} \rangle.\] 
The proof concludes by invoking Lemma~\ref{lem:Observation}.
\end{proof}

\subsection{Relative Dehn Functions}

Suppose that $\langle S,\mathcal{H}\mid R\rangle$ is a finite relative presentation of the pair $(G,\mathcal H)$.  For a word $W$ over  the alphabet $\mathcal{S}=S\sqcup \bigsqcup_{H\in\mathcal{H}}(H-\{1\})$ representing the trivial element in $G$, there is an expression
\begin{equation}\label{eq:relation}W=\prod_{i=1}^k f_i^{-1}R_i f_i\end{equation}
where $R_i\in R$ and $f_i\in F(S)$. We say a function $f\colon \mathbb{N}\to \mathbb{N}$ is a \emph{relative isoperimetric function} of the relative presentation $\langle S,\mathcal{H}\ |\ R \rangle$ if, for any $n\in \mathbb{N}$, and any word $W$ over  the alphabet $\mathcal{S}$ of length $\leq n$  representing the trivial element in $G$, one can write $W$ as in \eqref{eq:relation} with $k\leq f(n)$.
The smallest relative isoperimetric function of a finite relative presentation $\langle S,\mathcal{H}\ |\ R \rangle$ is called the \emph{relative Dehn function of $G$ with respect to $\mathcal{H}$}, or the \emph{Dehn function of the pair $(G, \mathcal H)$}. This function is denoted by  $\Delta_{G,\mathcal{H}}$.   Theorem~\ref{Thm:Osin:well:defined} below justifies the notation $\Delta_{G,\mathcal{H}}$ for the   Dehn function of a finitely presented pair $(G,\mathcal{H})$.

For functions $f,g\colon \mathbb{N}\to \mathbb{N}$, we write $f\preceq g$ if there exist constants $C,K,L\in \mathbb{N}$ such that
$f(n)\leq Cg(Kn)+Ln$ for every $n$. We say $f$ and $g$ are \emph{asymptotically equivalent}, denoted as $f \asymp g$, if $f\preceq g$ and $g\preceq f$.

\begin{theorem} \emph{\cite[Theorem 2.34]{Osin06}} \label{Thm:Osin:well:defined} 
Let $G$ be a finitely presented group relative to the collection of subgroups $\mathcal H$. Let $\Delta_1$ and $\Delta_2$ be the relative Dehn functions associated to two finite relative presentations. If $\Delta_1$ takes only finite values, then $\Delta_2$ takes only finite values, and $\Delta_1\asymp \Delta_2$.
\end{theorem}

The   Dehn function of a pair $(G,\mathcal H)$ is \emph{well-defined} if it takes only finite values. This can be characterized in terms of fine graphs as follows.  

\begin{definition}[Cayley-Abels graph for  pairs]
A \emph{Cayley-Abels graph} of the pair $(G,\mathcal{H})$ is a connected cocompact simplicial $G$-graph $\Gamma$ such that: 
\begin{enumerate}
    \item edge $G$-stabilizers are finite,
    
    \item vertex $G$-stabilizers are either finite or conjugates of subgroups in $\mathcal{H}$,
    
    \item every $H\in \mathcal{H}$ is the $G$-stabilizer of a vertex of $\Gamma$, and
    
    \item any pair of vertices of $\Gamma$ with the same $G$-stabilizer $H\in\mathcal H$ are in the same $G$-orbit if $H$ is infinite.
\end{enumerate}
 \end{definition}

 \begin{theorem}\label{thm:DehnFunctionFine}
Let  $(G,\mathcal{H})$ be a proper pair. The following statements are equivalent.
\begin{enumerate}
    \item The   Dehn function $\Delta_{G,\mathcal{H}}$ is well-defined.
    \item $(G,\mathcal H)$ is finitely  presented and there is a fine Cayley-Abels graph of $(G,\mathcal H)$.
    \item $(G,\mathcal H)$ is finitely  presented and every  Cayley-Abels graph of $(G,\mathcal H)$ is fine.
\end{enumerate}
\end{theorem}

 Theorem~\ref{thm:DehnFunctionFine} is essentially~\cite[Theorem E]{HuMaSa21} together with a  result on  Cayley-Abels graphs from ~\cite[Theorem H]{ArMP22}. This is described below. 
 
Concrete examples of Cayley-Abels graphs can be exhibited using the following construction introduced by Farb~\cite{Farb}, see also~\cite{Hruska}.

\begin{definition}[Coned-off Cayley graph]
Let $(G,\mathcal{H})$ be a pair, and let $S$ be a finite relative generating set of $G$ with respect to $\mathcal{H}$. Denote by $G/\mathcal{H}$ the set of all cosets $gH$ with $g\in G$ and $P\in \mathcal H$. The \emph{coned-off Cayley graph $\hat\Gamma(G,\mathcal{H},S)$} is the graph  with vertex set $G\cup G/\mathcal H$ and edges of the following type
\begin{itemize}
    \item $\{g,gs\}$ for $s\in S$ and $g\in G$,
    \item $\{x, gH\}$ for $g\in G$, $H\in \mathcal{H}$ and $x\in gH$.
\end{itemize}
\end{definition}

That a pair $(G,\mathcal H)$ has a well-defined function is characterized in terms of fineness of coned-off Cayley graphs.

\begin{theorem}\emph{\cite[Theorem E]{HuMaSa21}} \label{thm:SLE-ThmE0}
Let  $(G,\mathcal{H})$ be a finitely presented pair with a finite generating set $S$.  The Dehn function $\Delta_{G,\mathcal{H}}$ is well-defined if and only if the coned-off Cayley graph  $\hat\Gamma(G,\mathcal{H},S)$ is fine. \end{theorem} 

 Every coned-off Cayley graph $\hat\Gamma(G,\mathcal{H},S)$ with $S$ a finite relative generating set is a Cayley-Abels graph. The following result implies that Coned-off Cayley graphs are, up to quasi-isometry,  independent of the choice of finite generating set, and we denote them by  $\hat \Gamma(G, \mathcal{H})$.      Observe now that Theorem~\ref{thm:DehnFunctionFine} also follows from the following result. 
 
 \begin{theorem}\emph{\cite[Theorem H]{ArMP22}}\label{thmX:uniqueCAgraph}
If $\Gamma$ and $\Delta$ are Cayley-Abels graphs of the proper pair $(G,\mathcal H)$, then:
\begin{enumerate}
    \item $\Gamma$ and  $\Delta$ are quasi-isometric, and
    \item   $\Gamma$ is fine if and only if $\Delta$ is fine.
\end{enumerate}
\end{theorem}

\section{Combination Theorems for Graphs}\label{sec:03}
In this section, we state two technical results, Theorems~\ref{thm:CombinationFine} and~\ref{thm:HNNFine}, which will be proven in the subsequent sections. The section includes how to  deduce the main results of the article, Theorems~\ref{thmx:Io} and~\ref{thmx:IIIo}, from these technical results.   

\begin{thm}\label{thm:CombinationFine}
For $i\in\{
 1,2\}$, let $(G_i, \mathcal{H}_i\cup\{K_i\})$ be a pair and   $\partial_i\colon C\to K_i$  a group monomorphism. Let $G=G_1\ast_C G_2$  denote the amalgamated product determined by
$G_1\xleftarrow{\partial_1}C\xrightarrow{\partial_2}G_2$, and  $\mathcal{H}=\mathcal{H}_1\cup\mathcal{H}_2$.
Let $\Gamma_i$ be a $G_i$-graph that has a vertex $x_i$ with $G_i$-stabilizer $K_i$. Then there is a $G$-graph $\Gamma$ with the following properties:
\begin{enumerate}
    \item $\Gamma$ has a vertex $z$ such that the $G$-stabilizer  $G_z=\langle K_1,K_2\rangle$, and there is a $G_i$-equivariant inclusion $\Gamma_i\hookrightarrow \Gamma$ that maps $x_i$ to $z$.
    \item If $\Gamma_i$ is connected for $i=1,2$, then $\Gamma$ is connected. 

    \item If every $H\in\mathcal{H}_i\cup\{K_i\}$ is the $G_i$-stabilizer of a vertex of $\Gamma_i$ for $i=1,2$, then every $H\in \mathcal{H}\cup\{\langle K_1,K_2 \rangle\}$ is the $G$-stabilizer of a vertex of $\Gamma$.
    
        \item If vertex $G_i$-stabilizers in $\Gamma_i$  are finite or conjugates of subgroups in $\mathcal{H}_i\cup\{ K_i\}$ for $i=1,2$, then vertex $G$-stabilizers in $\Gamma$ are finite or conjugates of subgroups in $\mathcal H\cup\{ \langle K_1,K_2\rangle\}$.
        
        \item If $\Gamma_i$ has finite edge $G_i$-stabilizers for $i=1,2$, then $\Gamma$ has finite edge $G$-stabilizers.
        \item If $\Gamma_i$ has finitely many $G_i$-orbits of vertices (edges) for $i=1,2$, then $\Gamma$ has finitely many $G$-orbits of vertices (resp. edges).
    \item If $\Gamma_i$ is fine for $i=1,2$,   then $\Gamma$ is fine.
    \item If $\Gamma_i$ is fine at $V_\infty(\Gamma_i)$ for $i=1,2$, then $\Gamma$ is fine at $V_\infty(\Gamma)$.
    \item If $\Gamma_i$ is hyperbolic for $i=1,2$, then $\Gamma$ is hyperbolic.
     \item If $\Gamma_i$ is simplicial for $i=1,2$, then $\Gamma$ is simplicial.
\end{enumerate}
\end{thm}

Let us explain how Theorem~\ref{thmx:Io} follows from the above result.
 
\begin{proof}[Proof of Theorem~\ref{thmx:Io}]
For the first statement, suppose $\mathcal{H}_i\cup \{K_i\}$ is hyperbolically embedded in $G_i$. Then 
$\mathcal{H}_i\cup \{K_i\}$ is an almost malnormal collection of subgroups of $G_i$ by~\cite[Prop. 4.33]{DGO17}. In particular,   $(G_i, \mathcal{H}_i\cup \{K_i\})$ is a proper pair.  By Theorem~\ref{thm:Farhan}, there is a $(G_i, \mathcal{H}_i\cup \{K_i\})$-graph   $\Gamma_i$. Let $x_i$ be a vertex of $\Gamma_i$ with $G_i$-stabilizer $K_i$. Applying  Theorem~\ref{thm:CombinationFine} to $\Gamma_1$,  $\Gamma_2$, $x_1$ and $x_2$,  we obtain a $(G_1\ast_CG_2, \mathcal{H}\cup\{\langle K_1,K_2 \rangle\})$-graph.
Note that $(G_1\ast_CG_2, \mathcal{H}\cup\{\langle K_1,K_2 \rangle\})$ is a proper pair by a standard argument using normal forms.
Then invoke Theorem~\ref{thm:Farhan} to obtain that $\mathcal{H}\cup\{\langle K_1,K_2 \rangle\}$ is hyperbolically embedded in $G_1\ast_CG_2$.

The second statement is proved analogously.  Suppose the relative Dehn function of $(G_i,\mathcal{H}_i\cup\{K_i\})$ is well-defined. 
By~\cite[Prop. 2.36]{Os06b}, the pair $(G_i,\mathcal{H}_i\cup\{K_i\})$ is proper. It follows that $(G_1\ast_CG_2, \mathcal{H}\cup\{\langle K_1,K_2 \rangle\})$ is also a proper pair by a standard argument using normal forms.
By Theorem~\ref{thm:DehnFunctionFine},  $(G_i,\mathcal{H}_i\cup\{K_i\})$ is finitely presented and admits a fine Cayley-Abels graph $\Gamma_i$. In particular, there is a vertex $x_i\in \Gamma_i$ with $G_i$-stabilizer equal to $K_i$.
Apply Theorem~\ref{thm:CombinationFine} to $\Gamma_1$,   $\Gamma_2$ and the vertices $x_1,x_2$ to obtain a fine Cayley-Abels graph $\Gamma$ for the pair $(G_1\ast_CG_2, \mathcal{H}\cup\{\langle K_1,K_2 \rangle\})$. Since $(G_1\ast_CG_2, \mathcal{H}\cup\{\langle K_1,K_2 \rangle\})$ is finitely presented by Lemma~\ref{prop:AmalgamatedPresentation}, then Theorem~\ref{thm:DehnFunctionFine} implies that the relative Dehn function of $(G_1\ast_CG_2, \mathcal{H}\cup\{\langle K_1,K_2 \rangle\})$ is well-defined.
\end{proof}

\begin{thm}\label{thm:HNNFine}
Let $(G, \mathcal{H} \cup\{K,L\})$ be a pair with $K\neq L$, $C\leq K$, and  $\varphi\colon C\to L$ a group monomorphism. Let $G\ast_\varphi$ denote the HNN-extension
$\langle G, t\mid  t c t^{-1} =\varphi(c)~\text{for all $c\in C$} \rangle$.
Let $\Delta$ be a $G$-graph that has vertices $x$ and $y$ such that their $G$-stabilizers are $K$ and $L$ respectively, and their $G$-orbits are disjoint. Then there is a $G\ast_\varphi$-graph $\Gamma$ with the following properties:
\begin{enumerate}
    \item $\Gamma$ has a vertex $z$ such that $G_z=\langle K^t,L\rangle$, and there is a $G$-equivariant inclusion $\Delta\hookrightarrow \Gamma$ such that  $x\mapsto t^{-1}.z$ and $y\mapsto z$.
    
    \item If $\Delta$ is connected, then $\Gamma$ is connected. 

    \item If every $H\in\mathcal{H}\cup\{K,L\}$ is the $G$-stabilizer of a vertex of $\Delta$, then every $H\in \mathcal{H}\cup\{\langle K^t,L \rangle\}$ is the $G\ast_\varphi$-stabilizer of a vertex of $\Gamma$.
    
    \item If vertex $G$-stabilizers in $\Delta$  are finite or conjugates of subgroups in $\mathcal{H}\cup\{ K,L\}$, then vertex $G\ast_\varphi$-stabilizers in $\Gamma$ are finite or conjugates of subgroups in $\mathcal H\cup\{ \langle K^t,L\rangle\}$.
        
        \item If $\Delta$ has finite edge $G$-stabilizers, then $\Gamma$ has finite edge $G\ast_\varphi$-stabilizers.
        \item If $\Delta$ has finitely many $G$-orbits of vertices (edges), then $\Gamma$ has finitely many $G\ast_\varphi$-orbits of vertices (resp. edges).
    \item If $\Delta$ is fine,   then $\Gamma$ is fine.
    \item If $\Delta$ is fine at $V_\infty(\Delta)$, then $\Gamma$ is fine at $V_\infty(\Gamma)$.
    \item If $\Delta$ is hyperbolic, then $\Gamma$ is hyperbolic.
\end{enumerate}
\end{thm}

\begin{proof}[Proof of Theorem~\ref{thmx:IIIo}]
This proof is completely analogous to the proof of Theorem~\ref{thmx:Io}: invoke 
Theorem~\ref{thm:HNNFine}
and Lemma~\ref{prop:RelPreHNN}
instead of  Theorem~\ref{thm:CombinationFine} and Lemma~\ref{prop:AmalgamatedPresentation}
respectively. 
\end{proof}

\section{Amalgamated Products and Graphs}\label{sec:04}

This section describes an argument proving Theorem~\ref{thm:CombinationFine}. While the statement of this result seems intuitive, we are not aware of a full account of those techniques in a common framework, so this section provides a detailed construction.

\subsection{Pushouts in the category of $G$-sets.} Let $\phi\colon R\to S$ and $\psi\colon R \to T$ be $G$-maps. The pushout of $\phi$ and $\psi$ is defined as follows. Let $Z$ be the $G$-set  obtained as the quotient of 
the disjoint union of $G$-sets  $S\sqcup T$ by the equivalence  relation generated by all pairs   $s\sim t$ with $s\in S$ and $t\in T$ satisfying that  there is $r\in R$ such that $\phi(r)=s$ and $\psi(r)=t$.
There are canonical $G$-maps $\imath\colon S \to Z$ and $\jmath\colon T \to Z$ such that $\imath\circ \phi= \jmath \circ \psi$. This construction satisfies the universal property of pushouts in the category of $G$-sets.

\begin{proposition}\label{lem:pushout-set}
 Let $\phi\colon R\to S$ and $\psi\colon R \to T$ be $G$-maps. Consider the pushout 
  \[ 
  \begin{tikzcd} 
                    &  S \arrow{rd}{\imath} &  \\
R \arrow{ru}{\phi} \arrow{rd}{\psi} &     & Z \\                    &  T\arrow{ru}{\jmath} &
\end{tikzcd} \]
of $\phi$ and $\psi$. 
Suppose there is $r\in R$ such that $R=G.r$. If $s=\phi(r)$, $t=\psi(r)$ and $z=\imath(s)$ then the $G$-stabilizer $G_z$ equals the subgroup $\langle G_s,G_t \rangle$. 
\end{proposition}
\begin{proof}
Since $\imath$ and $\jmath$ are $G$-maps, $\langle G_s,G_t \rangle \leq G_z$. Conversely, let $g\in G_z$.  If $g\in G_s$ then $g\in \langle G_s, G_t\rangle$. Suppose $g\not \in G_s$.  

Let $r_0$ denote the element $r\in R$ in the statement, in particular, $s=\phi(r_0)$,   $t=\psi(r_0)$ and $R=G.r_0$. Since $\jmath (t)=\imath(g.s)$, the definition of $Z$ as a collection of equivalence classes in $S\sqcup T$ implies 
that  there is a sequence $r_0',r_1,r_1'\ldots , r_k, r_k'$ of elements of $R$ such that 
\[ t=\psi(r_0'),\ \phi(r_0')=\phi(r_1),\ \psi(r_1)=\psi(r_1'),\  \ldots,\ \psi(r_k)=\psi(r_k'),\ \phi(r_k')=g.s.\]
Let $s_i=\phi(r_{i-1} ')=\phi(r_i)$ and $t_i=\psi(r_i)=\psi(r_i')$.
Since $R=G.r_0$, there are elements $a_0,a_1,\ldots,a_k$ and $b_0,b_1,\ldots, b_{k-1}$ of $G$ such that 
\[ a_i.r_i=r_i' \quad \text{and} \quad b_j.r_j'=r_{j+1} \]
for $0\leq i\leq k$ and $0\leq j<k$. Then \[g.s=\phi(r_k')=\phi(a_k b_{k-1} a_{k-1} \ldots b_0a_0.r_0)= a_k b_{k-1} a_{k-1} \ldots b_0a_0.s \]
and hence $ a_k b_{k-1} a_{k-1} \ldots b_0a_0 \in gG_s.$
Since $G_s \leq \langle G_s, G_t \rangle$, to prove that $g\in \langle G_s, G_t \rangle$ is enough to show that $a_i, b_j \in \langle G_s, G_t \rangle$. We will  argue by induction. 

First note that since $\phi$ and $\psi$ are $G$-maps
\[ a_i.s_i = s_{i+1} \quad \text{and} \quad b_j.t_{j}=t_{j+1
},\]
and hence 
\[  G_{s_{i+1}}= a_iG_{s_i}a_i^{-1}   \quad \text{and} \quad G_{t_{j+1}}=b_jG_{t_{j}}b_j^{-1}.\]
Moreover,   $t_i=\psi(r_i)=\psi(r_i')=\psi(a_i.r_i)=a_i.t_i$ implies 
\[a_i\in G_{t_i},\]
and analogously 
$s_{j+1}=\phi(r_{j+1})=\phi(b_j.r_j')=b_j.s_{j+1}$ implies
\[b_j\in G_{s_{j+1}}.\]
 
Since $t_0=t$ and $s_0=s$, we have that \[a_0\in G_{t_0} \leq \langle G_s, G_t \rangle, \quad \text{and} \quad b_0\in G_{s_1}=a_0G_{s_0}a_0^{-1} \leq \langle G_s, G_t \rangle.\]
Suppose $i<k$, $a_i,b_i\in \langle G_s, G_t \rangle$,  $G_{s_i}\leq \langle G_s, G_t \rangle$ and $G_{t_i}\leq \langle G_s, G_t \rangle$. Then
\[ a_{i+1} \in G_{t_{i+1}} = b_iG_{t_i}b_i^{-1} \leq \langle G_s, G_t \rangle,\]
and hence
\[  G_{s_{i+1}}= a_iG_{s_{i}}a_i^{-1}  \leq \langle G_s, G_t \rangle.\]
In the case that $i+1<k$,  
\[ b_{i+1} \in G_{s_{i+2}} = a_{i+1}G_{s_{i+1}}a_{i+1}^{-1} 
\leq   \langle G_s, G_t \rangle. \]
Therefore, by induction, $a_i,b_j\in \langle G_s, G_t \rangle$ for $0\leq i\leq k$ and $0\leq j<k$.
\end{proof}

\subsection{Extending actions on sets.} In the case that $K$ is a subgroup of $G$ and $S$ is a $K$-set, one can extend the $K$-action on $S$ to a $G$-set $G\times_K S$ that we now describe.  Up to isomorphism of $K$-sets, we can assume that $S$ is a disjoint union of $K$-sets
\[ S  =  \bigsqcup_{i\in I} K/K_i \]
where  $K/K_i$ is the $K$-set consisting of left cosets of a subgroup $K_i$ of $K$.
Then the $G$-set $G\times_K S$ is defined as a disjoint union of $G$-sets
\[  G\times_K S := \bigsqcup_{i\in I} G/K_i .\]
Observe that the canonical $K$-map \[ \imath\colon S \to G\times_K S, \qquad K_i \mapsto K_i\]
is injective. This construction satisfies a number of useful properties that we summarize in the following proposition.

For $n$ a natural number and  a set $X$, let $[X]^n$ denote the collection of subsets of $X$ of cardinality $n$. If $X$ is a $G$-set, then $[X]^n$ is a $G$-set with action defined as $g.\{x_1,\ldots ,x_n\}=\{g.x_1,\ldots ,g.x_n\}$.

\begin{proposition}\label{lem:fromHsetstoGset}
Let $K\leq G$ and $S$ a $K$-set. 
\begin{enumerate}
\item The canonical $K$-map $ \imath\colon S \to G\times_K S$ induces a bijection of orbit spaces $S/K \to   (G\times_K S) /G$.\label{2.1-orbits}

\item For each $s\in S$, the $K$-stabilizer $K_s$ equals the $G$-stabilizer $G_{\imath(s)}$.\label{2.1-stabilizers}

\item If $T$ is a $G$-set and $f\colon S \to T$ is  $K$-equivariant, then there is a unique $G$-map $\tilde f\colon  G\times_K S \to T$ such that $\tilde f \circ \imath = f$.\label{2.1-universal} 
 
\item  If $\imath(S)\cap g.\imath(S)\neq \emptyset$ for $g\in G$, then $g\in K$ and $\imath(S)=g.\imath(S)$.
\label{2.1-disjunion}

\item  In part three, if $f$ induces an injective map $S/K \to T/G$ and  $K_s = G_{f(s)}$ for every $s\in S$, then $\tilde f$ is injective.

\item  Let $\jmath\colon [S]^n \to G\times_K[S]^n$ be the  canonical map. Then for every $n\in\mathbb N$, there is a $G$-equivariant injection $\hat \imath \colon G\times_K [S]^n \to [G\times_K S]^n$ such that $\hat \imath \circ \jmath = \bar \imath$ where
$\bar \imath \colon [S]^n \to [G\times_K S]^n$ is the natural $K$-map induced by $\imath\colon S \to G\times_K S$.  
\end{enumerate}
\end{proposition}

\begin{proof}
The first four statements  are observations.  
For the fifth statement, suppose $\tilde f(\imath(s_1) ) = \tilde f(g.\imath (s_2))$. 
Then $f( s_1 )=g.f( s_2 )$. 
Since the map $S/K\to T/G$ induced by $f$ is injective, we have that $s_1$ and $s_2$ are in the same $K$-orbit in $S$, say  $s_2=k.s_1$ for $k\in K$. It follows that
$f( s_1 )=gk.f( s_1 )$, and since $K_{s_1} = G_{f(s_1)}$, we have that $gk\in K_{s_1}$. Therefore 
$\imath(s_1) =\imath(gk.s_1)= g.\imath(ks_1) =g.\imath(s_2)$.

The sixth statement is proved as follows. The $K$-map 
$\imath\colon S \to G\times_K S$ naturally induces a $K$-map 
$\bar \imath \colon [S]^n \to [G\times_K S]^n$. By the third statement, there is a unique $G$-map $\hat \imath \colon G\times_K [S]^n \to [G\times_K S]^n$ such that $\hat  \imath  \circ \jmath =  \bar \imath $ where $\jmath\colon [S]^n \to G\times_K[S]^n$.  As a consequence of the fourth statement,  $\imath\colon [S]^n \to [G\times_K S]^n$ induces an injective map  $[S]^n/K \to [G\times_K S]^n/G$ and $K_A = G_{\imath(A)}$ for every $A\in [S]^n$; therefore $\hat \imath$ is injective.
\end{proof}

As the reader might have noticed, this construction is an instance of general categorical phenomena; that formulation will have no use in this article so we will not discuss it.

\subsection{Graphs as 1-dimensional complexes} While the objectives of this section only require us to consider simplicial graphs, the category of simplicial graphs does not have pushouts~\cite{St83}. For this reason, it is convenient to work within the framework of 1-dimensional complexes or equivalently graphs in the sense that we describe below. We will only consider a particular class of pushouts of graphs that behaves well over simplicial graphs.
A \emph{graph} is a triple $(V,E,r)$, where $V$ and $E$ are sets, and $r\colon E \to [V]^2$ is a function where $[V]^n$ is the collection of nonempty subsets of $V$ of cardinality at most $n$. Elements of the set  $V$ and $E$ are called \emph{vertices} and \emph{edges} respectively; the function $r$ is called the \emph{attaching map} . For a graph $\Gamma$, we denote $V(\Gamma)$ and $E(\Gamma)$ its vertex and edge set, respectively. If $v\in V(\Gamma)$, $e\in E(\Gamma)$ and $v\in r(e)$, then $v$ is \textit{incident} to $e$, and $v$ is called an \emph{endpoint} of $e$. Vertices incident to the same edge are called \textit{adjacent}. 

The graph $(V,E,r)$ is \emph{simplicial} if every edge has two distinct endpoints and $r$ is injective. Equivalently, $(V,E,r)$ is simplicial if $r\colon E \to [V]^2$ is injective and its image does not intersect $[V]^1$.

A graph $\Delta$ is a \emph{subgraph} of a graph $\Gamma$ if $V(\Delta) \subset V(\Gamma)$,   $E(\Delta) \subset E(\Gamma)$ and $r_\Delta$ equals the restriction of $r_\Gamma$ to $E(\Delta)$. Abusing notation, we consider any vertex of a graph $\Gamma$ as an edgeless subgraph with a single vertex, and any 
 edge $e$ of $\Gamma$  as the subgraph with vertex set the set of vertices incident to $e$ in $\Gamma$ and edge set consisting of only $e$. 

  For a vertex $u$ of a simplicial graph $\Gamma=(V,E,r)$, let  $\mathsf{star_\Gamma(u)}$ denote the subgraph with vertex set $V(\mathsf{star}(u))=\{u\}\cup\{v\in V\mid \text{$v$ is adjacent to $u$}\}$ and edge set $E(\mathsf{star}(u)) = \{e\in E\mid \text{the endpoints of $e$ belong to $V(\mathsf{star}(u))$}\}$ and the attaching map is the corresponding  restriction of $r$.

{Our notion of morphism allows the collapse of edges to single vertices. Specifically, a \emph{morphism} of $\phi\colon (V,E,r) \to (V',E',r')$ of  graphs is a pair of maps $\phi_0\colon V\to V'$ and $\phi_1\colon E \to  V'\cup E'$ such that there is a commutative diagram 
\[
\begin{tikzcd} 
\phi_1^{-1}(E') \arrow["\phi_1"]{r}\arrow["r"]{d} & E'\arrow["r'"']{d} \\
{[V]^2} \arrow["\phi_0"]{r} & {[V']^2}
\end{tikzcd} \quad \begin{tikzcd} 
\phi_1^{-1}(V') \arrow["\phi_1"]{r}\arrow["r"]{d} & V'\arrow["\cong"']{d} \\
{[V]^2} \arrow["\phi_0"]{r} & {[V']^1 }
\end{tikzcd}\]
where the horizontal bottom arrow $\phi_0$ is the natural $G$-map induced by $\phi_0\colon V\to V'$, and $V'\to [V']^1$ is the natural bijection given by  $v\mapsto \{v\}$. Observe that in general for a morphism $\phi=(\phi_0,\phi_1) \colon \Gamma \to \Delta$ of graphs, the map $\phi_0$ does not determine $\phi_1$; however if $\Delta$ is simplicial then $\phi_0$ determines $\phi_1$. A morphism $(\phi_0,\phi_1)$   is a \emph{monomorphism} (also called an \emph{embedding}) if both maps are injective.

 Given a graph morphism $\phi=(\phi_0, \phi_1)\colon \Gamma \to \Delta$ and a subgraph $\Theta$ of $\Delta$, the \emph{preimage} $\phi^{-1}(\Theta)$ is the subgraph of $\Gamma$ with vertex set $\phi_0^{-1}(V(\Theta))$ and edge set $\phi_1^{-1}(V(\Theta)\cup E(\Theta))$.

Let $G$ be a group. A \emph{$G$-graph} is a graph $(V,E,r)$ where $V$ and $E$ are $G$-sets, and $r$ is a $G$-map with respect to the natural $G$-action on $[V]^2$ induced by the $G$-set $V$. A morphism $(\phi_0,\phi_1)$ of $G$-graphs is a morphism of graphs such that each $\phi_i$ is a $G$-map. A \emph{$G$-equivariant embedding} is a monomorphisms of $G$-graphs. A $G$-action on a graph $\Gamma$ \emph{has no  inversions} if for every $e\in E$ and $g\in G$ such that $g.e=e$, $g.v=v$ for every  $v\in r(e)$.  For a $G$-action without inversions on a graph $\Gamma$ and  $K\leq G$, let $\Gamma^K$ denote  subgraph of $\Gamma$ defined by  $V(\Gamma^K) =\{v\in V(\Gamma)\mid k.v=v \text{ for all $k\in K$}\}$ and $E(\Gamma^K) =\{e\in E(\Gamma)\mid k.e=e \text{ for all $e\in K$}\}$. 

\subsection{Extending group actions on graphs}

Let $K$ be a subgroup of $G$, and let $\Lambda=(V,E,r)$ be a $K$-graph.
Define \[ G\times_K \Lambda = (G\times_K V, G\times_K E, \tilde r)\] where $\tilde r$ is unique $G$-map induced by the commutative diagram
\[ \begin{tikzcd} 
E \arrow["r"]{d}\arrow["\jmath"]{r}& G\times_K E \arrow["\tilde r"']{d} \\
{[V]^2} \arrow["\imath"]{r} &  { [G\times_K V]^2}
\end{tikzcd}\]
where $\imath\colon V \hookrightarrow G\times_K V$ and $\jmath\colon E \hookrightarrow G\times_K E$ are the canonical $K$-maps, see  Lemma~\ref{lem:fromHsetstoGset}\eqref{2.1-universal}.  
 Note that  there is a \emph{canonical $K$-equivariant embedding}  \[ \Lambda \hookrightarrow G\times_K \Lambda\] induced by $\imath$ and $\jmath$. We consider $\Lambda$ a $K$-subgraph of $G\times_K \Lambda$.    
 \begin{remark}\label{rem:ExtendingAction}
Lemma~\ref{lem:fromHsetstoGset}, parts \eqref{2.1-stabilizers} and \eqref{2.1-disjunion} imply:
\begin{enumerate}
\item If  $\Lambda$ is a simplicial $K$-graph without inversions, then $G\times_K\Lambda$ is a simplicial $G$-graph without inversions.
 
 \item \label{rem:connected}
For any connected subgraph $\Delta$ of $G\times_K \Lambda$, there is $g\in G$ such that $g.\Delta$ is a subcomplex of $\Lambda$, in a commutative diagram, 
\[ \begin{tikzcd} 
G\times_K\Lambda \arrow["g"]{r} & G\times_K \Lambda \\
\Delta \arrow[hookrightarrow]{u} \arrow["g"]{r} &  \Lambda \arrow[hookrightarrow,]{u}
\end{tikzcd}
\]
In particular, if $\Lambda$ is connected, then every connected component of $G\times_K\Lambda$ is isomorphic to $\Lambda$.
\end{enumerate}
\end{remark}

\subsection{Pushouts of graphs}
Let $\mathsf X$ and $\mathsf Y$ be $G$-graphs, let $C\leq G$ be a subgroup and suppose $\mathsf X^C$ and $\mathsf Y^C$ are non-empty. Let $x\in \mathsf X^C$ and $y\in \mathsf Y^C$ be vertices. The\emph{  {$C$-pushout $\mathsf Z$ of $\mathsf X$ and $\mathsf Y$ with respect to the pair $(x,y)$}} is the $G$-graph $\mathsf Z$
obtained by taking the disjoint union of $\mathsf X$ and $\mathsf Y$ and then identifying the vertex $g.x$ with the vertex $g.y$ for every $g\in G$.

Equivalently, the $C$-pushout $\mathsf Z$ of $\mathsf X$ and $\mathsf Y$ with respect to the pair $(x,y)$ is the $G$-graph $\mathsf Z$
whose vertex set $V(Z)$ is the pushout of the $G$-maps $\kappa_1\colon G/C\to \mathsf V(\mathsf X)$ and $\kappa_2\colon G/C\to \mathsf V(\mathsf Y)$   given by $C\mapsto x$ and $C\mapsto y$; and edge set the disjoint union of the $G$-sets $E(\mathsf X)$ and $E(\mathsf Y)$, and attaching map $E(\mathsf Z) \to V(\mathsf Z)^2$ defined as the union of the attaching maps for $\mathsf X$ and $\mathsf Y$ postcomposed with the maps $V(\mathsf X) \to V(\mathsf Z)$ and $V(\mathsf Y) \to V(\mathsf Z)$ defining the pushout.
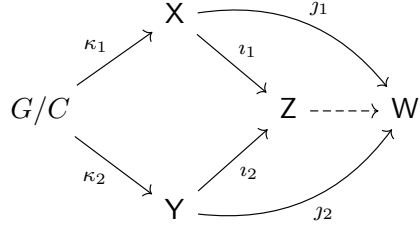
\[ 
  \begin{tikzcd} 
                    &   \mathsf X \arrow[rrd, bend left, "\jmath_1"] \arrow{rd}{\imath_1} & &  \\
G/C \arrow[ru, "\kappa_1"] \arrow[rd, "\kappa_2"' ]  &     & \mathsf Z\ar[r, dashed] & \mathsf W\\                    &  \mathsf Y\arrow[ ru, "\imath_2"'] \arrow[rru, bend right, "\jmath_2"'] & &
\end{tikzcd} \]
The standard universal property of pushouts holds for this construction: if $\jmath_1\colon \mathsf X \to \mathsf W$ and $\jmath_2\colon \mathsf Y \to \mathsf W$ are morphisms of $G$-graphs such that  $\jmath_1\circ \kappa_1 = \jmath_2\circ \kappa_2$, then there is a unique morphism of $G$-graphs $\mathsf Z\to \mathsf W$ such that above diagram commutes. 

\begin{remark}\label{rem:stabilizers} 
Let $\mathsf Z$ be the $C$-pushout of $\mathsf X$ and $\mathsf Y$ with respect to a pair $(x,y)$. \begin{enumerate}
\item For any vertex $x$ in $\mathsf X$,   $G_x = G_{\imath_1(x)}$ or $x$ is in the image of $\kappa_1$.  

\item For any edge $e$ of $X$, $G_e = G_{\imath_1(e)}$.

\item If  $\mathsf X / G$ and $\mathsf Y / G$ both have  finitely many     vertices (resp. edges), then $\mathsf Z / G$ has finitely many vertices (resp. edges). 
\end{enumerate}
\end{remark}

\begin{exmp}
Let $G=A\ast_C B$ where $A$ and $B$ are free abelian groups of rank two, and $C$ is maximal cyclic subgroup of $A$ and $B$. Let $\mathsf X$ be the $A$-graph consisting of a single vertex with the trivial $A$-action, and define $\msf Y$ analogously for $B$. Then the graph $G\times_{A} \mathsf X$ is the edgeless $G$-graph with vertex set the collection of left cosets of $G/A$; and analogously $G\times_{B} Y$ is the edgeless graph with vertex set $G/B$. Let $\msf Z$ be the $C$-pushout of $\msf X$ and $\msf Y$. By parts (4) and (7) of Proposition~\ref{lem:pushout},  $\msf Z$ is a connected edgeless $G$-graph and hence  a single vertex. (Note that the algebraic nature of $A$, $B$ and $C$ was not used in the argument).
\end{exmp}

\begin{exmp}
 Let $A=\langle a_1,a_2,a_3\mid [a_1,a_2] \rangle$ and $B=\langle b_1,b_2,b_3\mid [b_1,b_2] \rangle$, and let   $\mathsf X=\hat\Gamma(A,\langle a_1,a_2 \rangle, a_3)$ and $\mathsf Y=\hat\Gamma(B,\langle b_1,b_2 \rangle, b_3)$ be the coned-off Cayley graphs. Note that $\msf X$ is the Bass-Serre tree of the splitting of $A$ as the graph of groups 
  
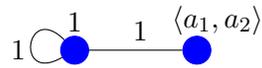
\[
\begin{tikzpicture}
	\begin{pgfonlayer}{nodelayer}
		\node [style=blue] (0) at (-7, 0) {};
		\node [style=blue] (1) at (-3.75, 0) {};
		\node [style=none] (2) at (-7, 0.75) {$1$};
		\node [style=none] (3) at (-5.25, 0.5) {$1$};
		\node [style=none] (5) at (-3.2, 0.8) {$\langle a_1, a_2\rangle$};
		\node [style=none] (6) at (-8.5, 0) {$1$};
	\end{pgfonlayer}
	\begin{pgfonlayer}{edgelayer}
		\draw (0) to (1);
		\draw [in=-150, out=135, loop] (0) to ();
	\end{pgfonlayer}
\end{tikzpicture}

\]
with two vertices and two edges with trivial edge group.

Let $G=A\ast_C B$ be the amalgamated product where $C$ corresponds to the cyclic subgroup $\langle a_1 \rangle \leq A$ and $\langle b_1 \rangle \leq B$. Consider the $C$-pushout $\msf Z$ of $G\times_A \msf X$ and $G\times_B \msf Y$. By the fourth, fifth and sixth statements of Proposition~\ref{lem:pushout} below, $\msf Z$ is a tree, it contains three distinct $G$-orbits of vertices, two of these $G$-orbits have all representatives with trivial stabilizer, and there is a  vertex $z$ with stabilizer $\langle a_1, a_2, b_2\rangle = \langle a_1, a_2\rangle \ast_{\langle a_1 \rangle=\langle b_1 \rangle} \langle b_1, b_2 \rangle$, and there are four  distinct orbits of edges all with representatives having trivial stabilizer. Hence $\msf Z$ is the Bass-Serre tree of a splitting of $G$ given by the graph of groups
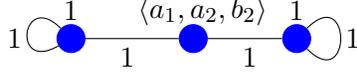
\[
\begin{tikzpicture}
	\begin{pgfonlayer}{nodelayer}
		\node [style=blue] (0) at (-7, 0) {};
		\node [style=blue] (1) at (-3.75, 0) {};
		\node [style=none] (2) at (-7, 0.75) {$1$};
		\node [style=none] (3) at (-5.5, -0.5) {$1$};
		\node [style=none] (5) at (-3.5, 0.75) {$\langle a_1, a_2, b_2\rangle$};
		\node [style=none] (6) at (-8.5, 0) {$1$};
		\node [style=blue] (7) at (-1, 0) {};
		\node [style=none] (8) at (0.5, 0) {$1$};
		\node [style=none] (9) at (-1, 0.75) {$1$};
		\node [style=none] (10) at (-2.25, -0.5) {$1$};
	\end{pgfonlayer}
	\begin{pgfonlayer}{edgelayer}
		\draw (0) to (1);
		\draw [in=-150, out=135, loop] (0) to ();
		\draw (1) to (7);
		\draw [in=45, out=-45, loop] (7) to ();
	\end{pgfonlayer}
\end{tikzpicture}

\]
with three vertices and four edges. In particular, $Z$ is the coned-off Cayley graph of $\hat\Gamma(G, G_y, \{a_3,b_3\})$. 
 \end{exmp}

\begin{proposition}\label{lem:pushout}
Let $G$ be the amalgamated free product group $A\ast_C B$, let  $\mathsf{X}$ be a $A$-graph, and let $\mathsf{Y}$ be a $B$-graph. Let $x\in \mathsf{X}^C$    and $y\in \mathsf{Y}^C$ be vertices. Let $\mathsf{Z}$ be the $C$-pushout of $G\times_{A} \mathsf{X}$ and $G\times_{B} \mathsf{Y}$  with respect to $( x, y)$. Let $z=\imath_1(x)=\imath_2(y)$. The following properties hold:
\begin{enumerate}
 \item The homomorphism $A_x \ast_C B_y \to G$ induced by the inclusions $A_x\leq G$ and $B_y\leq G$ is injective and has image $G_z$. In particular $G_z=\langle A_x, B_y \rangle$ is isomorphic to $A_x \ast_C B_y$.\label{married}

    \item The morphism   $\mathsf{X}\hookrightarrow G\times_A \mathsf{X} \xrightarrow{\imath_1} \mathsf{Z}$
    is an $A$-equivariant embedding. Analogously, 
    $\mathsf{Y}\hookrightarrow G\times_B \mathsf{Y} \xrightarrow{\imath_2} \mathsf{Z}$
    is a $B$-equivariant embedding. 
    
 \item[] From here on, we consider $\mathsf X$ and $\mathsf Y$ as subgraphs of $\mathsf Z$ via these canonical embeddings.

   \item For every vertex $v$ (resp. edge $e$) of $\mathsf Z$, there is $g\in G$ such that $g.v$  is a vertex (resp. is an edge $g.e$) of the subgraph $\mathsf X \cup \mathsf{Y}$. \label{transfer}

 \item For every vertex $v$   of $\mathsf{X}$ which is not in the $A$-orbit of $x$, $A_v = G_v$ where $G_v$ is the $G$-stabilizer of $v$ in $\mathsf Z$. Analogously for every vertex $v$ of  $\mathsf{Y}$ not in the $B$-orbit of $y$, $B_v=G_v$.\label{singles}

\item For every edge $e$ of  $\mathsf X$ (resp. $\mathsf Y$), $A_e=G_e$ (resp. $B_e=G_e$) where $G_e$ is the $G$-stabilizer of $e$ in $\mathsf Z$ .\label{edges}

    \item If the complexes $\mathsf X / A$ and $\mathsf Y / B$ both have  finitely many 
    vertices (resp. edges), then $\mathsf Z / G$ has finitely many vertices (resp. edges).\label{cocompact}

  \item If $\mathsf{X}$ and $\mathsf{Y}$ are connected, then $\mathsf{Z}$ is connected.\label{connected}

    \item  \label{cutpoint}
  There is a  $G$-tree $T$ and a morphism $\xi\colon \mathsf Z \to T$ of $G$-graphs with the following properties:
    The $G$-orbit of $\xi(z)$ and its complement in the set of vertices of $T$ make $T$ a bipartite graph;  
    the preimage $\xi^{-1}(\xi(z))$ is a single vertex; and if a vertex $v$ of $T$ is not in the $G$-orbit of $\xi(z)$, then the preimage of the star of $v$ is a subgraph of $\mathsf Z$ isomorphic to $\mathsf X$ or $\mathsf Y$.  
\end{enumerate}
\end{proposition}

\begin{proof}
The first item is a direct consequence of Proposition~\ref{lem:pushout-set}. For the second item,  first note that  that the composition   $\mathsf X\hookrightarrow G\times_A \mathsf X \xrightarrow{\imath_{1}} \mathsf Z$ is a morphism of $A$-graphs. Observe that to prove the embedding part  is enough to consider only vertices of $\mathsf X$ that are in the $A$-orbit of $x$. Suppose that $a.x$ and $x$ with $a\in A$ both map to $z \in \mathsf Z$. Then $a \in G_{z} = A_{x} \ast_C B_y$ and therefore $a\in A_x$ and hence $a.x=x$.   Item three follows directly from the definition of $\mathsf Z$, and items four to six are consequences of  Proposition~\ref{lem:fromHsetstoGset}. 

To prove the seventh statement suppose that $X$ and $Y$ are connected graphs.  The subgraph $\mathsf X\cup \mathsf Y$ of $\mathsf Z$ is connected since both $\mathsf X$ and $\mathsf Y$  contain the vertex $z$. On the other hand, any vertex of $\mathsf Z$ belongs to a translate of $\mathsf X \cup \mathsf Y$ by an element of $G$. Therefore to prove that $\mathsf Z$ is connected,  it is enough to show that for any $g\in G$ there is a path in $\mathsf Z$ from $z$ to $g.z$.  For any $g\in G$ and $a\in A$, there is a path from $g.z$ to $ga.z$ in $\mathsf Z$: indeed, there is a path from $z$ to $a.z$ in the connected $A$-subgraph $\mathsf X$ of $\mathsf Z$, and hence there is a path from $g.z$ to $ga.z$ in $\msf Z$. Analogously, for any $g\in G$ and $b\in B$, there is a path from $g.z$ to $gb.z$. Since any element of $G$ is of the form $a_1b_1\dots a_nb_n$ with $a_i\in A$ and $b_i \in B$, there is a path from $z$ to $g.z$ for any $g\in G$.

Now we prove the eighth statement. Observe that $G$ splits as $G=A\ast_{A_x}(A_x\ast_CB_y)\ast_{B_y}B$ where the subgroups $A_x$, $B_y$ and $A_x\ast_C B_y$ are naturally identified with the $G$-stabilizers of $x\in G\times_A \msf X$, $y\in G\times_B \msf Y$, and $z\in \msf Z$. Let $T$ denote the Bass-Serre tree of this splitting. The vertex and edge sets of $T$ can be described as 
\[ V(T) = G/A \sqcup G/(A_x\ast_C B_y)  \sqcup G/B \]
and 
\[ E(T) = \{ \{ gA, g(A_x\ast_C B_y) \} \mid g\in G \} \sqcup \{ \{g(A_x\ast_C B_y), gB\} \mid g\in G \}\]
respectively. Note that $T$ is a bipartite $G$-graph, the equivariant bipartition of the vertices given by $G/A\sqcup G/B$ and $G/(A_x\ast_CB_y)$. 

Consider the $A$-map from $\mathsf X$ to $T$ that maps every vertex of $\mathsf X$ not in the $A$-orbit of $x$ to the vertex $A$, and $x\mapsto A_x\ast_C B_y$. Since $T$ is simplicial, this induces a unique morphism of $G$-graphs $\jmath_1\colon G\times_A \mathsf X \to T$. Analogously, there is $B$-map $\mathsf Y\to T$ that maps every vertex not in the $B$-orbit of $y$ to the vertex $B$ and $y\mapsto A_x\ast_C B_y$; this induces a unique $G$-map $\jmath_2\colon G\times_B \mathsf Y \to T$. 
\[ 
  \begin{tikzcd} 
                    &   G\times_A X \arrow[rrd, bend left, "\jmath_1"] \arrow{rd}{\imath_1} & &  \\
G/C \arrow[ru, "\kappa_1"] \arrow[rd, "\kappa_2"' ]  &     & Z\ar[r, dashed, "\xi"] & T\\                    &  G\times_B Y\arrow[ ru, "\imath_2"'] \arrow[rru, bend right, "\jmath_2"'] & &
\end{tikzcd} \] 

Consider the $G$-maps $\kappa_1 \colon G/C \to G\times_A \mathsf X$ and $\kappa_2\colon G/C \to G\times_B \mathsf Y$ given by $C\mapsto x$ and $C\mapsto y$ respectively. Since $\jmath_1\circ \kappa_1 = \jmath_2\circ \kappa_2$,  the universal property implies that there is a surjective $G$-map $\xi\colon \mathsf Z\to T$. 

Note that $\xi^{-1}(\xi(z))$ is contained in the orbit $G.z$. Suppose $g.z\in \xi^{-1}(\xi(z))$. Then $g(A_x\ast_C B_y) = A_x\ast_C B_y$ and hence $g\in A_x\ast_C B_y$. Since $A_x\ast_C B_y$ is the $G$-stabilizer of $z$, we have that $g.z=z$. This shows that $\xi^{-1}(\xi(z))=\{z\}$.

Let us conclude by proving that if 
$v\in V(T)$ is not in the $G$-orbit of $\xi(z)=A_x\ast_c B_y$ then $\xi^{-1}(\mathsf{star_T}(v))$ is a graph isomorphic to either $\mathsf X$ or $\mathsf{Y}$. Note that   such a vertex $v$ is an element of $G/A \cup G/B$. By equivariance, it is enough to consider the two symmetric cases, namely $v=A$ or $v=B$. Let us prove that $\xi^{-1}(\mathsf{star}_T{A})$ is isomorphic to $\mathsf X$. 
Observe that any edge of  $\mathsf{star}_T(A)$ is of the form $\{A, a(A_x\ast_C B_y)\}$ with $a\in A$. Since 
$(\xi\circ \imath_1)^{-1}(\mathsf{star}_T(A)) = \jmath_1^{-1}(\mathsf{star}_T(A))$ and 
$\xi^{-1}(\mathsf{star}_T(A))  \subset \imath_1(G\times_A X)$, we have that  $\xi^{-1}(\mathsf{star}_T(A)) = \imath_1(\jmath_1^{-1}(\mathsf{star}_T(A))$.
Recall that the canonical $A$-map  $\mathsf{X} \to G\times_A \mathsf{X}$ is injective and this defines a natural identification of $\mathsf{X}$ with a subgraph of $G\times_A \mathsf{X}$ which equals $\jmath^{-1}(\mathsf{star}_T(A))$ by definition of $\jmath$. Then we have that  $\imath_1(\jmath_1^{-1}(\mathsf{star}_T(A))=\imath_1(\mathsf{X})$ is isomorphic to $\mathsf X$ by the second item of the proposition.
\end{proof}

\subsection{Proof of Theorem~\ref{thm:CombinationFine}}

\begin{lemma}\label{lem:FineAndHyp}
Let $\xi\colon \Gamma \to T$ be a morphism of graphs where 
$T$ is a bipartite tree, say $V(T)=K \cup L$. Suppose $\xi^{-1}(v)$ is a single vertex for every  $v\in L$, and $\xi^{-1}(\mathsf{star}(v))$ is a connected subgraph for every $v\in K$.  Let $\Omega = \{ \xi^{-1}(\mathsf{star}(v)) \mid v\in K  \}$. Then:  
\begin{enumerate}
\item \label{lem:simplicial} $\Gamma$ is a simplicial graph if and only if every $\Delta\in \Omega$ is simplicial.
    
    \item $\Gamma$ is a $\delta$-hyperbolic graph if and only if  every $\Delta\in \Omega$ is a $\delta$-hyperbolic graph.
    \item For any vertex $u$  of $\Gamma$, the following statements are equivalent:
    \begin{itemize}
        \item     $\Gamma$ is fine at $u$.
        \item For every $\Delta \in \Omega$, if $u$ is a vertex of $\Delta$, then $\Delta$ is fine at $u$.
    \end{itemize}
\end{enumerate}
\end{lemma}
\begin{proof}
Note for any vertex $u$ of $\Gamma$, if $\xi(u)\in L$ then $u$ is a cut vertex of $\Gamma$. The bipartite assumption on $T$ implies that if $\Theta$ is the closure of a connected component of $\Gamma\setminus \xi^{-1}(L)$, then $\Theta$ equals some $\Delta \in \Omega$.    

The first and second statements follow from the previous observation, the second one with important   generalizations~\cite{BeFe92}.  For the third  statement, if $\Gamma$ is fine at $u$, then any subgraph containing $u$ is fine at $u$. Conversely, let $u$ be a vertex of $\Gamma$ such that any $\Delta$ containing $u$ is fine at $u$. There are two cases to consider.  

Suppose that $\xi(u) \in K$. Then   there is a unique $\Delta\in\Omega$ that contains $u$. The bipartite assumption implies that $T_u \Delta = T_u\Gamma$. Since every vertex of $\Gamma$ that maps to $L$ disconnects $\Gamma$,  the metric spaces $(T_u \Delta, \angle_u)$ and $(T_u \Gamma, \angle_u)$ coincide. Since $\Delta$ is fine at $u$, then $\Gamma$ is fine at $u$. 

Suppose that $\xi(u)\in L$.  Observe  if $x,y\in T_u\Gamma$ and $x$ and $y$ belong to different subgraphs in $\Omega$, then $\angle_u(x,y)=\infty$. 
Therefore, for any $x\in T_u\Gamma$,  every ball of finite radius in $T_u\Gamma$ centered at $x$ is a ball of finite radius in $T_u \Delta$ centered at $x$ for some $\Delta$. Since by assumption, $\Delta$ is fine at $v$, every ball of finite radius in $T_v\Gamma$ centered at $x$ is finite. 
\end{proof}

\begin{proof}[Proof of Theorem~\ref{thm:CombinationFine}]
 Let $\Gamma$ be the $C$-pushout of the $G$-graphs  $G\times_{G_1} \Gamma_1$
and $G\times_{G_2} \Gamma_2$ with respect to $(x_1,x_2)$, and let $z$ be the image of $x_1$ in $\Gamma$. The first six properties of $\Gamma$ are direct corollaries of  Proposition~\ref{lem:pushout}. The last four properties follow directly by invoking   Proposition~\ref{lem:pushout}\eqref{cutpoint} and  Lemma~\ref{lem:FineAndHyp}.
\end{proof}

\section{HNN-Extensions}\label{sec:05}

This section describes a proof of Theorem~\ref{thm:HNNFine}. The argument is analogous to  the one proving Theorem~\ref{thm:CombinationFine}. In this case, we need to construct a {$G\ast_\varphi$-graph} from a given $G$-graph that we call the $\varphi$-coalescence.

 \begin{definition}\label{defn-coalescence}(Coalescence in sets)
 Let $H$ be a subgroup of a group $A$,  let $\varphi\colon H\to A$ be a monomorphism and let $G$ be the HNN-extension
 \[ G=A\ast_{\varphi}=\langle G, t\mid  t c t^{-1} =\varphi(c)~\text{for all $c\in C$} \rangle .\]  Let $ X$ be an $A$-set, $x\in X^H$ and $y\in  X^{\varphi(H)}$. The {\emph{ $\varphi$-Coalescence of $X$ with respect to $(x,y)$}} is the $G$-set $Z$ arising as quotient  of $G\times_A  X$ by the equivalence relation generated by the set of basic relations
 \[ \mathcal{B}=\{ (gt.x, g.y) \mid  \text{ $g\in G$} \}. \] 
 Note that the quotient map \[ \rho\colon G\times_A X \to Z \] is $G$-equivariant. 
\end{definition}

\begin{exmp}
Let $\varphi\colon A\to A$ be a group automorphism and consider the HNN-extension $G=A\ast_\varphi$. Let $X$ be the $A$-set consisting of a single point. Then $G\times_A X$ is the $G$-space $G/A$, and then the $\varphi$-Coalescence of $X$ is again a single point.
\end{exmp}

\begin{exmp}
Consider a free product $A=H_1\ast H_2$.  Let $\varphi\colon H_1 \to H_2$ be an isomorphism, and $G=A\ast_\varphi$. Let  $X$ be the $A$-set  $A/H_1 \cup A/H_2$ of all left cosets of $H_1$ and $H_2$ in $A$. Then $G\times_A X$ is the $G$-set of left cosets $G/H_1 \cup G/H_2$. The  $\varphi$-coalescence $Z$ of $ X$ with respect to the pair $(H_1, H_2)$ is the quotient $G\times_A  X$ by identifying $gtH_1$ and $gH_2$ for every $g\in G$. Hence $Z$ is naturally isomorphic as a $G$-set to $G/H_1$. Observe that the $A$-map $X \to Z$ given by $H_1\mapsto H_1$ and $\varphi(H_1)\mapsto tH_1$ is an injective $A$-map.
\end{exmp}

\begin{lemma}\label{basic chain}
Let $H$ be a subgroup of a group $A$,  let $\varphi\colon H\to A$ be a monomorphism and let $G=A\ast_{\varphi}$.  Let $ X$ be an $A$-set, let $x,y\in X$ be in different $A$-orbits such that $A_x=H$,   $A_y=\varphi(H)$.   If
$Z$ is the $\varphi$-Coalescence of $ X$ with respect to $(x,y)$, and $z=\rho(y)$, then:
\begin{enumerate}
    \item the $G$-stabilizer $G_z$ equals $\varphi(H)$, and  \item the $A$-map $\jmath\colon  X \to  Z$ defined by the commutative diagram
\[ \begin{tikzcd} 
 & G\times_A  X \arrow["\rho"]{rd} &  \\
 X \arrow["\jmath"]{rr} \arrow["\imath", hookrightarrow]{ru}  & &  Z, 
\end{tikzcd}
\]
is injective.
\end{enumerate}
\end{lemma}

\begin{proof}
The inclusion $\varphi(H) \subseteq G_z$ is a consequence of $\rho$ being  $G$-equivariant, $\varphi(H)=A_y$ and $\rho(y)=z$. Conversely, let $g\in G_z$. Then $g.y\sim y$ in $G\times_A X$, and it follows that there is an integer $n\geq1$ and a sequence $w_1, w_2,\ldots, w_n$ of elements of $G\times_A X$ such that $g.y=w_1$, $w_n=y$ and $w_i$ and $w_{i+1}$  are the components of a basic relation (see the definition of coalescence). Since $x$ and $y$ are in  different $A$-orbits in $X$, they represent different $G$-orbits in $G\times_A X$, see the first item of  Proposition~\ref{lem:fromHsetstoGset}.
Hence, we have that $w_i=g_i.x$ if $i$ is even, and $w_i=g_i.y$ if $i$ is odd, for some elements $g_i$ of $G$ where $g_1=g$ and $g_n\in\phi(H)$. Note that the integer $n$ is odd, and the  {chain of basic relations between the $w_i$'s} in $G\times _A X$ can be expressed as
\[ g_1.y \sim g_2.x \sim g_3.y\sim g_4.x \sim \dots \sim g_{n-1}.x \sim g_n.y.\]
By definition of basic relation,  $tg_2^{-1}g_1\in \varphi(H)$, $tg_2^{-1}g_3\in \varphi(H)$, $tg_4^{-1}g_3\in \varphi(H)$, $tg_4^{-1}g_5\in \varphi(H)$ and so on until $tg_{n-1}^{-1}g_n\in \varphi(H)$. Since $n$ is odd, we have that 
\[g_1^{-1}g_n=(tg_2^{-1}g)^{-1}(tg_2^{-1}g_3)(tg_4^{-1}g_3)^{-1}(tg_4^{-1}g_5)\dots  (tg_{n-1}^{-1}g_n)\in \varphi(H),\]
which implies $g=g_1\in  \varphi(H)$. This shows that $G_z=\varphi(H)$
 
Now we prove the second statement. By Lemma~\ref{lem:fromHsetstoGset}, the natural $A$-map $ X \to G\times_A X$ is injective. Observe that $Z$ is obtained as a quotient of $G\times_A  X$ by the $G$-equivariant equivalence relation generated by the \emph{basic relation} $t.x \sim y$. 
Hence to prove injectivity of
$X\to G\times_A  X \to Z$, it  is enough to show that
the restriction to $A.x \cup A.y$ is injective. 
Assume there are $a_1,a_2\in A$ such that $a_1.x$ and $a_2.x$ map to the same element in $ Z$. 
Then letting $a=a_2^{-1}a_1$, both $a.x$ and $x$ map to the same element in $ Z$. 
Hence  $a.x\sim x$ which implies that $at^{-1}.y\sim t^{-1}.y$. Therefore $(ta^{-1}t^{-1}).y\sim y$ and thus by the first statement, $ta^{-1}t^{-1}\in \varphi(H)$, and hence $a^{-1}\in t^{-1}\varphi(H)t=H$. 
This results in $a\in H$. 
Therefore $a_2^{-1}a_1\in H$ and $a_1.x=a_2.x$.
We have shown that the restriction $A.x\to  Z$ is injective. With a similar argument one can show that $A.y \to Z$ is also injective. 
\end{proof}

\begin{definition}\label{defn-coalescence2}(Coalescence in graphs)
Let $H$ be a subgroup of a group $A$,  let $\varphi\colon H\to A$ be a monomorphism, and let $G=A\ast_{\varphi}$.  Let $\msf X$ be an $A$-graph, let $x,y\in V(\msf X)$  such that $x\in X^H$ and $y\in  X^{\varphi(H)}$.  The \emph{$\varphi$-Coalescence $\msf Z$ of $\msf X$ with respect to $(x,y)$} is the $G$-graph with vertex set  the $\varphi$-Coalescence of the $A$-set $V(\msf X)$ with respect to $(x,y)$, edge set $G\times_A E(\msf X)$, and attaching map $E(\msf Z) \to [V(Z)]^2$ defined as the composition
\[ 
  \begin{tikzcd} 
  E(\msf X) \arrow[r]\arrow[d] & {[V(\msf X)]}^2 \arrow[d] \\
  E(\msf Z)=G\times_A E(\msf X) \arrow[r] \arrow[rd]   & G\times_A [V(\msf X)]^2 \arrow[d]\\
&  {[V(\msf Z)]}^2 
\end{tikzcd} \]
where the horizontal middle arrows are induced by the attaching map $E(\msf X) \to [V(\msf X)]^2$ (see Lemma~\ref{lem:fromHsetstoGset}\eqref{2.1-universal}) and the the bottom vertical map is induced by the quotient map $G\times_AV(X) \to V(Z)$. Let $\rho\colon G\times_A \msf X \to \msf Z$ denote the induced $G$-morphism.  
\end{definition}  

 \begin{remark}[Equivalent definition of coalescence]
Equivalently, the $\varphi$-Coalescence  $\mathsf{Z}$ of the $A$-graph $\msf X$ with respect to $(x,y)$ is the quotient $\msf Z$ of the $G$-graph  $G\times_A \msf X$ by the equivalence relation generated by $gt.x \sim g.y$ for all $g\in G$.
\end{remark}

\begin{proposition}\label{lem:HNN}
Let $H$ be a subgroup of a group $A$,  let $\varphi\colon H\to A$ be a monomorphism and let $G=A\ast_{\varphi}$.  Let $\msf X$ be an $A$-graph, let $x,y\in \msf X$ in different $A$-orbits such that $A_x=H$,   $A_y=\varphi(H)$.   If
$\msf Z$ is the $\varphi$-Coalescence of $\msf X$ with respect to $(x,y)$, and $z=\rho(y)$,
 then the following properties hold:

\begin{enumerate}
 \item  $G_z=\varphi(H)$.\label{identified vertices-HNNN}

    \item The map  $\mathsf{X}\hookrightarrow G\times_A \mathsf{X} \xrightarrow{} \mathsf{Z}$
    is an $A$-equivariant embedding.
    
 \item[] From here on, we consider $\mathsf X$  as a subgraph of $\mathsf Z$ via this canonical embedding.

 \item For every vertex $v$ (resp. edge $e$) of $\mathsf Z$, there is $g\in G$ such that $g.v$  is a vertex (resp. is an edge $g.e$) of  $\mathsf X$. \label{transfer-HNN-}
 
\item For every vertex $v$   of $\mathsf{X}$ which is not in the $A$-orbit of $x$, $A_v = G_v$ where $G_v$ is the $G$-stabilizer of $v$ in $\mathsf Z$.

\item For every edge $e$ of  $\mathsf X$  $A_e=G_e$  where $G_e$ is the $G$-stabilizer of $e$ in $\mathsf Z$.\label{vertex and edges}

\item If the complex  $\mathsf X / A$ has  finitely many 
    vertices (resp. edges), then $\mathsf Z / G$ has finitely many vertices (resp. edges).\label{cocompact-HNN}

  \item If $\mathsf{X}$  is connected, then $\mathsf{Z}$ is connected.\label{connected-HNN}

    \item   There is a  $G$-tree $T$ and a morphism $\xi\colon \mathsf Z \to T$ of $G$-graphs with the following properties:
    The $G$-orbit of $\xi(z)$ and its complement in the set of vertices of $T$ make $T$ a bipartite $G$-graph;  
    the preimage $\xi^{-1}(\xi(z))$ is a single vertex; and if a vertex $v$ of $T$ is not in the $G$-orbit of $\xi(z)$, then the preimage of the star of $v$ is a subgraph of $\mathsf Z$ isomorphic to $\mathsf X$. \label{cutpoint-HNN}
 \end{enumerate}
\end{proposition}
The following argument is analogous to the proof of Proposition~\ref{lem:pushout}.
\begin{proof}
 The first and second statements are direct consequences of Lemma~\ref{basic chain} when considering $V(\msf X)$ and $E(\msf X)$  as $A$-sets. 
Items three to six  follow directly from the definition of $\mathsf Z$ and Proposition~\ref{lem:fromHsetstoGset}.    

Suppose $\mathsf X$ is connected. By Proposition~\ref{lem:fromHsetstoGset}, the graph $G\times_A \msf X$ is a  disjoint union of copies of the connected subgraph $\msf X$, and hence any element in $\msf Z$ belongs to a translate of $\msf X$ by an element of $G$. Therefore, to prove that $\msf Z$ is connected, it is sufficient to show that for any $g\in G$ there is a path in $\msf Z$ from $z$ to $g.z$.

First observe that if there is a path from $z$ to $g.z$, then there is a path from $z$ to $gt.z$. Indeed, there is a path from $x$ to $y$ in the connected subgraph $\msf X$ of $G\times_A \msf X$, and hence there is a path from $z=\rho(t.x)$ to $t.z=\rho(t.y)$ in $\msf Z$. Therefore, there is a path from $g.z$ to $gt.z$ in $\msf Z$, and in particular a path from $z$ to $gt.z$.

Now observe that if there is a path from $z$ to $g.z$, then there is a path from $z$ to $ga.z$ for any $a\in A$. Indeed,  there is a path from $z$ to $a.z$ in the connected $A$-subgraph $\mathsf X$ of $\mathsf Z$. Hence,  there is a path from $g.z$ to $ga.z$ in $\msf Z$, and in particular a path from $z$ to $ga.z$.

To conclude, any $g\in G$  is a product of the form $g=a_1t^{\theta_1}a_2t^{\theta_2}\dots a_nt^{\theta_n}a_{n+1}$ with $a_i\in A$ and $\theta_i=\pm 1$. Therefore, an induction argument using the two previous statements shows that there is a path from $z$ to $g.z$ in $\msf Z$ for every $g\in G$.

Now we prove the eighth statement. Consider the barycentric subdivision $T$ of the Bass-Serre tree  of the splitting $G\ast_{\varphi}$. Specifically, $T$ is the tree with vertex set
\[ V(T) = G/A \sqcup  G/H \]
edge set
\[ E(T) = \{ \{ gA, gtH \} \mid g\in G \} \sqcup \{ \{gA, gH\} \mid g\in G \}.\] 
Note that all the edges of $T$ are $G$-translates of the following two edges attached at the vertex $tH$,
\[
\begin{tikzpicture}
	\begin{pgfonlayer}{nodelayer}
		\node [circle,fill,inner sep=1pt] (0) at (-7, 0) {};
		\node [circle,fill,inner sep=1pt] (1) at (-4, 0) {};
		\node [style=none] (2) at (-7, 0.4) {$A$};
		\node [style=none] (3) at (-5.5, 0.75) {$\{A,tH\}$};
		\node [style=none] (5) at (-4, 0.4) {$tH$};
		\node [style=none] (6) at (-8.25, 0) {};
		\node [circle,fill,inner sep=1pt] (7) at (-1, 0) {};
		\node [style=none] (8) at (0.25, 0) {};
		\node [style=none] (9) at (-1, 0.4) {$tA$};
		\node [style=none] (10) at (-2.25, 0.75) {$\{tH, tA\}$};
	\end{pgfonlayer}
	\begin{pgfonlayer}{edgelayer}
		\draw (0) to (1);
		
		\draw (1) to (7);
	
	\end{pgfonlayer}
\end{tikzpicture} \]
Suppose that the $A$-set
$V(\msf X)=\left(\bigsqcup_{i\in  I}A/A_i\right) \sqcup A/H \sqcup A/\varphi(H)$. Then
\[V(G\times_A \msf X)=\big(\bigsqcup_{i\in I }G/A_i\big)\sqcup G/H\sqcup G/\varphi(H).\]
Since $T$ is a simplicial graph,   there is an induced  $G$-equivariant  morphism of graphs \[\psi\colon G\times_A \msf X   \to T\] defined on vertices by
\[A_i \mapsto A, \qquad H\mapsto H, \qquad \varphi(H) \mapsto tH.  \]
Note that any edge in $G\times_A \msf X$ of the form $\{gA_i,gaA_j\}$ for $g\in G$ and $a\in A$ is collapsed to the vertex $gA$ in $T$; and edges of the form $\{gA_i, gH\}$ and $\{gA_i, gtH \}$ are mapped to the edges $\{gA, gH\}$ and $\{gA, g\varphi(H)\}$ of $T$ respectively.  This map induces a  $G$-equivariant morphism of graphs  $\xi \colon \msf Z\to T$ such that the following diagram commutes, 
\[ \begin{tikzcd} 
 & G\times_A \msf X \arrow["\psi"]{dr} \arrow[swap,"\rho"]{dl}&  \\
\msf Z \ar[dashed, "\xi"]{rr}   & & T. 
\end{tikzcd}
\]
Indeed this diagram commutes since $\psi$ is $G$-equivariant and $\psi(tH)=\psi(\varphi(H))$.  

Observe that $\xi(z)=\psi(tH)$ and $\psi^{-1}(tH)=\{tH,\varphi(H)\}$, then $\rho(tH)=\rho(\varphi(H))$ implies that $\xi^{-1}(\xi(z))=\{z\}$. 

By definition of $T$, if a vertex $v$ is not in the $G$-orbit of $\xi(z)=tH$, then $v\in G/A$. Hence the partition $V(T)=G/A\sqcup G/H$` makes $T$ a bipartite $G$-graph. 
Any edge of $\mathsf{star}_T(A)$ is of the form $\{A,atH\}$ for some $a\in A$. Hence $\xi^{-1}(\mathsf{star}_T(A))=\rho(\psi^{-1}(\mathsf{star}_T(A))) = \rho ( \mathsf{X})$ and then the definition of $Z$ implies that $\rho ( \mathsf{X})$ is isomorphic $\mathsf{X}$. 
\end{proof}

\begin{proof}[Proof of Theorem~\ref{thm:HNNFine}]
The argument is the same as   the one used to prove Theorem~\ref{thm:CombinationFine}, the only difference is the use of  Proposition~\ref{lem:HNN} instead of Proposition~\ref{lem:pushout}.
\end{proof}

\section{Dehn functions and coarse isoperimetric functions}
\label{sec:addendum}

In this section, we recall the definition of coarse isoperimetric function of a graph and recall how   one can recover the relative Dehn function of a pair via Cayley-Abels graphs, see Theorem~\ref{thm:CoarseDehn}.  In the second part of the section, we discuss a technical result that provides bounds for coarse isoperimetric functions of graphs based on maps into trees, see Proposition~\ref{prop:TreeDehn}. These  results  are  used  to obtain bounds on the relative Dehn function of  fundamental groups of graph of groups based on the relative Dehn functions of the vertex groups, see Corollary~\ref{cor:DehnFunctions}.  

\subsection{Coarse isoperimetric functions}

A singular combinatorial map $X\to Y$ between 1-dimensional CW-complexes is a continuous map such that the restriction to each open 1-dimensional cell of $X$ is either a homeomorphism onto an open cell of $Y$ or its image is contained in the 0-skeleton of $Y$.  A singular combinatorial loop $c\colon I\to X$ is a singular combinatorial map  such that its domain is a CW-complex homeomorphic to a closed interval. The length $\mathsf{Len}(c)$ of $c$ is defined as the number of open 1-cells of $I$  that map homeomorphically to open cells of $X$. 

Let $\Gamma$ be a connected graph, regard it as a CW-complex, and consider the path-metric on $\Gamma$ obtained by regarding each edge as a segment of unit length. Let $k>0$. An \emph{$k$-filling} of a singular combinatorial loop $c\colon I \to \Gamma$ is a pair $(P,\Phi)$ consisting of a triangulation $P$ of the  $2$-disc $D^2$ and a singular combinatorial map $\Phi\colon P^{(1)}\to \Gamma$  such that $\Phi|_{\partial D}$ equals the closed path $c$ (after identifying the endpoints of the domain of $c$) and the image under $\Phi$ of the boundary of each 2-cell of $P$ is a set of diameter at most $k$ in $\Gamma$.  Define $|(P,\Phi)|$ to be the number of faces of $P$ and
\[\mathsf{area}^\Gamma_k(c):=\min\{|(P,\Phi)|\ \colon \ (P,\Phi) \text{ an $k$-filling of $c$}\}. \]
The \emph{$k$-coarse isoperimetric function $f_k^\Gamma\colon\mathbb{N}\to\mathbb{N}$} of $\Gamma$ is then defined to be
\[f^\Gamma_k(\ell):=\sup\{\mathsf{area}^\Gamma_k(c)\ \colon\ \mathsf{Len}(c)\leq \ell\}. \]
We say that $f^\Gamma_k$ is \emph{well-defined} if it takes only finite values. 
The graph $\Gamma$ is \emph{$k$-fillable} if $f_k^\Gamma$ is well-defined, and $\Gamma$ is \emph{fillable} if it is $k$-fillable for some integer $k$. Note that if $f_k^\Gamma$ is well-defined then $f_\ell^\Gamma$ is well-defined for all $\ell\geq k$. 

For two functions $f,g\colon \mathbb{N}\to \mathbb{N}$, define $f\preceq g$ if there exist constants $C,K,L\in \mathbb{N}$ such that
\[f(n)\leq Cg(Kn)+Ln.\]
We say that $f$ is \emph{asymptotically equivalent} to $g$ if $f \asymp g$ if $f\preceq g$ and $g\preceq f$.

\begin{proposition}\emph{\cite[Proposition~III.H.2.2]{BrHa99}}\label{Prop:Bridson:qi:invariance}
If $\Gamma$ and $\Gamma'$ are quasi-isometric connected graphs such that $\Gamma$ is fillable, then $\Gamma'$ is fillable  and $f_k^{\Gamma}\asymp f_k^{\Gamma'}$  for all sufficiently large integers $k$.
\end{proposition}

We conclude the subsection recalling two results in order to deduce Corollary~\ref{thm:CoarseDehn} which shows that the relative Dehn function of a finitely presented pair is equivalent to coarse isoperimetric fuctions of Cayley-Abels graphs. 

The following theorem is a re-statement of a result of Osin, see~\cite[Prop. 4.8]{HuMaSa21}. 
\begin{theorem}\cite[Thm. 2.53]{Osin06}\label{prop:Osin}
Let $G$ be a group and let $\mathcal{H}$ be a collection of subgroups. If $\Delta_{G,\mathcal{H}}$ is well-defined, then $\hat\Gamma(G,\mathcal{H})$ is fillable and $\Delta_{G,\mathcal{H}} \asymp f^{\hat\Gamma(G,\mathcal{H})}_k$ for all sufficiently large integers $k$.
\end{theorem}

\begin{theorem}\cite[Theorem E]{HuMaSa21}\label{thm:SLE-ThmE} 
Let $(G,\mathcal{H})$ be a finitely generated pair.  If $\hat\Gamma(G,\mathcal{H})$ is  fine and fillable, then
$(G, \mathcal{H})$ is finitely presented and  $\Delta_{G,\mathcal{H}}$ is well-defined.
\end{theorem}

As previously observed, the coned-off Cayley graphs of a   finitely generated pair $(G,\mathcal{H})$ are Cayley-Abels graphs of the pair. Theorem~\ref{thmX:uniqueCAgraph} states that all Cayley-Abels  graphs of a finitely generated pair are quasi-isometric, and if one of them is fine then all of them are fine. Moreover,  fillable and the class of coarse isoperimetric functions are quasi-isometry invariants of graphs by Proposition~\ref{Prop:Bridson:qi:invariance}.  Putting these results together with the two results above, and Theorem~\ref{thm:SLE-ThmE0}, one obtains the following corollary.

\begin{corollary}\label{thm:CoarseDehn}
 Let $\Gamma$ be a Cayley-Abels graph of finitely generated proper pair $(G,\mathcal{H})$.
\begin{enumerate}
\item If $\Delta_{G,\mathcal{H}}$ is well-defined, then $\Gamma$ is fine and fillable, and $\Delta_{G,\mathcal{H}} \asymp f^{\Gamma}_k$ for all sufficiently large integers $k$.
\item  If $\Gamma$ is  fine and fillable, then $(G, \mathcal{H})$ is finitely presented and  $\Delta_{G,\mathcal{H}}$ is well-defined.
\end{enumerate}
\end{corollary}

\subsection{Relative Dehn functions and Splittings }

Let $g\colon\mathbb{N} \to \mathbb{N}$ be a function.  Then $g$ is \emph{superadditive} if $g(m)+g(n)\leq g(m+n)$.   If $g(0)=0$ then the \emph{super-additive closure} $\overline{g}\colon\mathbb{N} \to \mathbb{N}$ of $g$ is the function 
\[ \overline{g}(n) = \max\left\{ \sum_{i=1}^k g(n_i) \mid k\in\mathbb{N},\ n_i\in\mathbb{N},\ \sum_{i=1}^kn_i=n \right\},\]
and it is an observation that $\bar g$ is the least super-additive function such that $g(n)\leq \overline{g}(n)$ for all $n$. Note that the requirement  $g(0)=0$  is necessary in order for $\bar g$ to be well-defined.    An outstanding open question raised by Mark Sapir is whether the Dehn function of any finite presentation is asymptotically equivalent to a superadditive function~\cite{GuSa99}. 

\begin{proposition}\label{prop:DehnRetraction}
Let $r\colon \Gamma \to \Delta$ be a retraction of graphs. If $\Gamma$ is $k$-fillable, then $\Delta$ is $k$-fillable and $f_k^\Delta(n)\leq f_k^\Gamma(n)$.
\end{proposition}
\begin{proof}
Let $c\colon I\to \Delta$ is a singular combinatorial loop. If $(P,\Phi)$ is a $k$-filling of $c$ in $\Gamma$ then it is an observation that $(P,r\circ\Phi)$ is a $k$-filling of $c$ in $\Delta$. Therefore $\mathsf{area}_k^\Delta(c)\leq \mathsf{area}_k^\Gamma(c)$ and the result follows.
\end{proof}

The following proposition is the main technical result of the section. 

\begin{proposition}\label{prop:TreeDehn}
Let $\xi\colon \Gamma \to T$ be a morphism of graphs where 
$T$ is a bipartite tree, say $V(T)=K \cup L$. Suppose $\xi^{-1}(v)$ is a single vertex for every  $v\in L$, and $\xi^{-1}(\mathsf{star}(v))$ is a connected subgraph for every $v\in K$.
Let $\Omega = \{ \xi^{-1}(\mathsf{star}(v)) \mid v\in K  \}$.

If there is $k>0$ such that each $\Delta\in\Omega$ is a $k$-fillable graph and 
\[g(n):= \sup\{f_k^\Delta(n)\mid  \Delta\in\Omega\} < \infty  \text{ for every $n$,}\] 
then $\Gamma$ is $k$-fillable 
and
\[f^\Gamma_k (n)\leq \overline{g}(n)\]
where $\overline{g}$ denotes the super-additive closure of the function $g\colon\mathbb{N}\to\mathbb{N}$.  
\end{proposition}
\begin{proof}
It is an observation that if $c_1$ and $c_2$ are singular combinatorial loops in $\Gamma$ with the same initial point, and both admit $k$-fillings, then the concatenated loop $c_1\cdot c_2$ admits a $k$-filling and
\[ \mathsf{Len}(c_1\cdot c_2) = \mathsf{Len}(c_1)+\mathsf{Len}(c_2)  \quad \text{and} \quad \mathsf{area_k}(c_1\cdot c_2) \leq  \mathsf{area_k}(c_1) + \mathsf{area_k}(c_2).\]

To prove that $f_k^\Gamma(n)\leq \overline{g}(n)$, we prove that if $c\colon I \to \Gamma$ is a singular combinatorial loop in $\Gamma$ then $\mathsf{area}^\Gamma_k(c)\leq \overline{g}(\mathsf{Len}(c))$.
  
Let $c\colon I \to \Gamma$ be a singular combinatorial loop in $\Gamma$.  Consider the loop $\xi\circ c$ in the tree $T$. The image of $\xi\circ c$ is a finite subtree $T_c$ of $T$. Let $\#T_c\cap K$ denote the number of vertices of $T_c$ that belong to $K$. To the loop $c$ assign the complexity $|c|=(\#T_c\cap K, \mathsf{Len}(c)) \in \mathbb{N}\times\mathbb{N}$.  Consider the lexicographical order on $\mathbb{N}\times\mathbb{N}$, and recall that this is well-ordered set.
We prove by induction on $(\#T_c\cap K, \mathsf{Len}(c))$  that $\mathsf{area}^\Gamma_k(c)\leq \overline{g}(\mathsf{Len}(c))$.  

Base case $|c|=(0,m)$.  Suppose that $T_c$ does not contain vertices in $K$. In this case the bipartite assumption on $T$ implies that $T_c$ consists of a single vertex $v$ in $L$. Since $\xi^{-1}(v)$ is a single vertex of $\Gamma$ it follows that $c$ is constant path and hence $\mathsf{Len}(c)=0$ and $\mathsf{area}^\Gamma_k(c)=0\leq \overline{g}(0)$.

Base case $|c|=(1,m)$. Suppose that the vertex set of $T_c$ contains a single vertex in $K$, say $v$. Then the bipartite assumption on $T$ implies that $T_c$ is a subgraph of $\mathsf{star}(v)$ and hence the image of $c$ is contained in the subgraph $\Delta=\xi^{-1}(\mathsf{star}(v))$. 
By assumption, $\Delta$ is $k$-fillable, and hence there is $k$-filling of $c$ in $\Delta$ which is trivially also a $k$-filling in $\Gamma$.  Hence \[\mathsf{area}_k^\Gamma(c)\leq \mathsf{area}_k^\Delta(c)\leq f^\Delta_k(\mathsf{Len}(c)) \leq g(\mathsf{Len}(c)) \leq \overline{g}(\mathsf{Len}(c)).\]

General case $|c|=(n, m)$ with $n\geq2$. For the inductive step, suppose that $T_c\cap K$ has at least two vertices in $K$. Without loss of generality, we can identify the domain $I$ of $c$ with the closed interval $[0,1]$ (with some CW-structure). Since $T_c$ is connected, the bipartite assumption on $T$, implies that $T_c$ contains a vertex $v\in L$ such that $v$ is not a leaf of $T_c$, in particular,  $T_c-\{v\}$ has at least two connected components. 
Then $(\xi\circ c)^{-1}(T_c-\{v\})$ is a disconnected open subset of $[0,1]$. Let $J_1$ be the closure of a connected component of $(\xi\circ c)^{-1}(T_c-\{v\})$.  By changing the initial point of the loop $c\colon I \to \Gamma$, we can assume that $J_1=[0,\alpha]$ for some $\alpha<1$. Let $J_2=[\alpha,1]$, and let $c_i$ be the restriction of $c$ to the interval $J_i$. Then $c_i$ is singular combinatorial loop, and $c$ is the concatenation $c_1\cdot c_2$.   Since $T_c-\{v\}$ is disconnected, it follows that $0<\mathsf{Len}(c_i)<\mathsf{Len}(c)$. Since $\#T_{c_i}\cap K \leq \#T_c\cap K$, it follows that $|c_i|<|c|$. Hence by induction  \[ \begin{split} \mathsf{area}^\Gamma_k(c)  & \leq \mathsf{area}^\Gamma_k(c_1) +\mathsf{area}^\Gamma_k(c_2) \\ & \leq  \overline{g}(\mathsf{Len}(c_1))+ \overline{g}(\mathsf{Len}(c_2)) \\
& \leq \overline{g}(\mathsf{Len}(c_1) + \mathsf{Len}(c_2)) = \overline{g}(\mathsf{Len}(c)) 
\end{split}\]
where the first inequality follows from the observation in the first paragraph of this proof, the second inequality uses the induction hypothesis, and the third one uses that $\overline{g}$ is superadditive. 
\end{proof}

\subsection{Proof of Theorem~\ref{cor:DehnFunctions}}

\begin{proof} 
The proofs of the first two statements are analogous, and the argument goes back to the method of proof of the corresponding theorems in the introduction. The  third statement is a consequence of the second one, see the proof of Corollary~\ref{thmx:II}\eqref{item:IIbb}. 

We prove the first statement and leave the proof of the second statement to the reader. We remark that the argument essentially re-proves Theorem~\ref{thmx:Io}\eqref{item:Ib}.

Let $\Gamma_i$ be a Cayley-Abels graph of $(G_i, \mathcal{H}_i  \cup\{K_i\})$ for $i=1,2$. Then $\Gamma_i$ has a vertex $x_i$ with $G_i$-stabilizer $K_i$.  Let $\Gamma$ be the $C$-pushout of $G\times_{G_1} \Gamma_1$ and $G\times_{G_2} \Gamma_2$
with respect to $(x_1,x_2)$.  Theorem~\ref{thm:CombinationFine} implies that $\Gamma$ is a Cayley-Abels graph of $(G_1\ast_C G_2, \mathcal{H}  \cup\{\langle K_1,K_2 \rangle\})$. By Proposition~\ref{lem:pushout}\eqref{cutpoint}, there is a morphism of graphs $\xi\colon \Gamma \to T$  that satisfies the hypothesis of Proposition~\ref{prop:TreeDehn}, namely, $T$ is a bipartite tree with $V(T)=K\cup L$   such that  $\xi^{-1}(v)$ is a single vertex for each $v\in L$, and $\xi^{-1}(\mathsf{star}(v))$ is isomorphic to $\Gamma_i$ for some $i=1,2$ for every $v\in K$. Corollary~\ref{thm:CoarseDehn} implies that  $\Gamma_1$ and $\Gamma_2$ are both $k$-fillable for some $k$. Then  Proposition~\ref{prop:TreeDehn} implies that $\Gamma$ is $k$-fillable and  
\[ f_k^\Gamma \preceq \overline{\max\{f_k^{\Gamma_1}, f_k^{\Gamma_2}\}} \asymp 
\max\left\{\overline{f_k^{\Gamma_1}}, \overline{f_k^{\Gamma_2}}\right\} .\]   
 Then Corollary~\ref{thm:CoarseDehn} implies 
that 
\[ \Delta \preceq  \max\left\{\overline{\Delta_1 }, \overline{\Delta_2} \right\}.\] 

On the other hand, the properties of the morphism $\Gamma\to T$ imply that there is a retraction  $\Gamma\to \Gamma_i$ and hence Proposition~\ref{prop:DehnRetraction} implies that 
$f_k^{\Gamma_i} \preceq f_k^\Gamma$ 
and therefore 
$ \Delta_i \preceq \Delta.$
\end{proof}

 \bibliographystyle{alpha}
 \bibliography{refs}

\end{document}